\documentclass[a4paper,11pt,reqno]{amsart}
\usepackage{amssymb}
\usepackage{amsmath}
\usepackage{enumitem}
\usepackage{mathrsfs}
\usepackage[all]{xy}
\usepackage{mathtools}
\usepackage{hyperref}

\usepackage[OT2,T1]{fontenc}
\usepackage[british]{babel}

\usepackage[margin=1.3in]{geometry}

\numberwithin{equation}{section} \numberwithin{figure}{section}

\DeclareMathOperator{\Pic}{Pic} 
\DeclareMathOperator{\Gal}{Gal} \DeclareMathOperator{\NS}{NS}
 \DeclareMathOperator{\Spec}{Spec}
 \DeclareMathOperator{\rank}{rank}
\DeclareMathOperator{\Hom}{Hom} \DeclareMathOperator{\re}{Re}
\DeclareMathOperator{\im}{Im}   
\DeclareMathOperator{\vol}{vol} 
 \DeclareMathOperator{\Val}{Val}

\DeclareMathOperator{\Br}{Br} \DeclareMathOperator{\Brnr}{Br_{nr}}

\DeclareMathOperator{\inv}{inv} \DeclareMathOperator{\res}{\partial}
 
 \DeclareMathOperator{\Sub}{Sub}
 \DeclareMathOperator{\Ind}{Ind}
\DeclareMathOperator{\HH}{H}

\DeclareSymbolFont{cyrletters}{OT2}{wncyr}{m}{n}
\DeclareMathSymbol{\Sha}{\mathalpha}{cyrletters}{"58}
\DeclareMathSymbol{\Be}{\mathalpha}{cyrletters}{"42}
\newcommand{\thorn}{\textit{\th}}

\newcommand{\OO}{\mathcal{O}}
\newcommand{\GL}{\textrm{GL}}
\newcommand{\PGL}{\textrm{PGL}}

\newcommand{\kbar}{\overline{k}}

\newcommand{\Fbar}{\overline{F}}

\newcommand{\Xbar}{\overline{X}}

\newcommand{\Tbar}{\overline{T}}
\newcommand{\LL}{\mathcal{L}}

\newcommand{\Norm}[2]{\operatorname{N}_{{{{#1}}/{{#2}}}}}
\newcommand{\Res}[2]{\operatorname{R}_{{{{#1}}/{{#2}}}}}
\newcommand{\Fourier}[1]{#1^{\wedge}}
\newcommand{\Adele}{\mathbf{A}}

\newcommand{\sbf}{\mathbf{s}}
\newcommand{\A}{\mathscr{A}}
\newcommand{\br}{\mathscr{B}}
\newcommand{\R}{\mathscr{R}}
\newcommand{\Denef}{\Delta}

\newcommand\PP{\mathbb{P}}
\newcommand\ZZ{\mathbb{Z}}
\newcommand\NN{\mathbb{N}}
\newcommand\QQ{\mathbb{Q}}
\newcommand\RR{\mathbb{R}}
\newcommand\CC{\mathbb{C}}
\newcommand\GG{\mathbb{G}}

\newcommand\Gm{\GG_\mathrm{m}}

\newtheorem{lemma}{Lemma}
\newtheorem{conjecture}[lemma]{Conjecture}
\newtheorem{theorem}[lemma]{Theorem}
\newtheorem{corollary}[lemma]{Corollary}

\theoremstyle{definition}
\newtheorem{example}[lemma]{Example}
\newtheorem{definition}[lemma]{Definition}
\newtheorem{remark}[lemma]{Remark}
\newtheorem*{assumption}{Assumption}
\newtheorem*{philosophy}{Philosophy}

\numberwithin{lemma}{section}

\title[Number of varieties which contain a rational point]
{The number of varieties in a family which contain a rational point}
\author{\sc Daniel Loughran}
\address{Daniel Loughran \\
Leibniz Universit\"{a}t Hannover,
Institut f\"{u}r Algebra, Zahlentheorie
    und Diskrete Mathematik\\
Welfengarten 1\\
30167 Hannover\\
Germany.}
\email{loughran@math.uni-hannover.de}
\subjclass[2010]
{14G05 (primary), 
11D45, 
14F22, 
14M25 
(secondary)}

\begin{document}

\begin{abstract}
	We consider the problem of counting the number of varieties
	in a family over a number field which contain a rational point.
	In particular, for products of Brauer-Severi varieties and
	closely related counting functions associated to Brauer group elements.
	Using harmonic analysis on toric varieties, we provide a positive answer
	to a question of Serre on such counting functions in some cases.
	We also formulate some conjectures on generalisations of Serre's problem.
\end{abstract}

\maketitle

\thispagestyle{empty}

\tableofcontents

\section{Introduction} \label{sec:intro}
Given a variety over a number field $F$, a fundamental problem in number
theory is to determine whether or not it contains a rational point.
More generally one may consider a \emph{family} of varieties over $F$,
given as the fibres of some morphism $\pi:Y \to X$, for which we have
the following fundamental questions.
\begin{enumerate}
	\item \label{Question:1} Is there a member of this family which contains a rational point?
	\item \label{Question:2} Is the set of fibres which contain a rational point infinite, or even Zariski dense?
	\item \label{Question:3} Can one give precise quantitative estimates for the distribution of such fibres?
\end{enumerate}
The focus of this paper will be on families of Brauer-Severi varieties,
i.e.~varieties which become isomorphic to a projective space over an algebraic
closure $\Fbar$ of $F$. The simplest example of a Brauer-Severi variety is a smooth conic, and already here
the problem is non-trivial. Indeed, it is unknown whether an arbitrary conic bundle
over $\PP_F^1$ with at least one rational point necessarily has a Zariski dense set of rational points
(see \cite{BMS14} for recent results on this).
To make Question $(\ref{Question:3})$ more precise, we shall use height functions.
To any embedding $X \subset \PP^n$ of $X$ over $F$ and any $x=(x_0:\cdots:x_n) \in X(F)$,
we may associated a height function via
\begin{equation} \label{eq:height}
    H(x)= \prod_{v \in \Val(F)}\max\{|x_0|_v,\ldots,|x_n|_v\}.
\end{equation}
More generally, there is a theory of height functions associated to adelically
metrised line bundles (see \S \ref{sec:Heights}).
We define the associated counting function to be
\begin{equation} \label{def:conic_counting_function}
    N(X,H,\pi,B) = \#\{x \in X(F): H(x) \leq B, x \in \pi(Y(F))\}.
\end{equation}
A point $x \in X(F)$ is counted by this only if
the fibre $\pi^{-1}(x)$ contains a rational point.

If every variety in the family contains a rational point, then the counting function
$N(X,H,\pi,B)=N(X,H,B)$ is independent of $\pi$ and simply counts the  rational points
of bounded height on $X$.
Manin and others \cite{FMT89,BM90} have formulated conjectures on the asymptotic
behaviour of such counting functions in special cases (e.g.~Fano varieties). It is conjectured
that if $X$ is Fano with $X(F)$ Zariski dense in $X$ and $H$ is an anticanonical height function,
then there exists $U \subset X$ open and $c_{U,H}>0$ such that
\begin{equation} \label{eqn:Manin}
	N(U,H,B) \sim c_{U,H}B(\log B)^{\rho(X) -1}, \quad \text{ as } B \to \infty,
\end{equation}
where $\rho(X) = \rank \Pic X$. One considers open subsets
to avoid ``accumulating subvarieties'' whose contribution may dominate the counting
problem (e.g.~lines on cubic surfaces). This conjecture has been proven for various Fano varieties
(e.g.~flag varieties \cite{FMT89}, certain del Pezzo surfaces \cite{BB11}
and certain complete intersections \cite{Bir62})
and for other varieties with sufficiently positive 
anticanonical bundle (e.g.~toric varieties \cite{BT98}).
Nevertheless this conjecture is false as stated and there are now counter-examples
over any number field \cite{BT96, Lou13}.
One of the aims of this paper is to try to generalise Manin's conjecture
to the counting functions (\ref{def:conic_counting_function}).

Such counting functions (\ref{def:conic_counting_function}) have been considered before.
For example, let $Y_{d,n}$ denote the total space of the family
of all smooth hypersurfaces over $\QQ$ of degree $d$ in $\PP^n_\QQ$ with $n,d\geq2$.
There is a natural projection $\pi_{d,n}:Y_{d,n} \to \PP^N_\QQ$ given by
the coefficients of each hypersurface, where $N=\smash{\binom{n+d}{d}} -1$.
Under the assumptions that   $d \leq n$, that $(d,n) \neq (2,2)$
and that the Brauer-Manin obstruction is the only one to the Hasse principle for such
hypersurfaces, Poonen and Voloch \cite{PV04} have shown that
$$N(\PP^N_\QQ,H,\pi_{d,n},B) \sim c_{d,n}B, \quad \text{ as } B \to \infty,$$
for some constant $c_{d,n} >0$. Here $H$
denotes the anticanonical height function on $\PP_\QQ^N$ given by the
$(N+1)$-st power of the height function (\ref{eq:height}). Note that this result implies
that a positive proportion of all such hypersurfaces contain a
rational point. Other families have also recently been considered
(e.g Ch\^{a}telet surfaces \cite{BB14a} and certain principal
homogeneous spaces under coflasque tori \cite{BB14b}),
where it was again shown that a positive proportion of the varieties under consideration
contain a rational point.

Results of this type do not hold in the case where $(d,n)=(2,2)$, i.e.~the case of conics.
Serre \cite{Ser90} has shown that for the family of all plane conics we have
$$N(\PP^5_\QQ,H,\pi_{2,2},B) \ll \frac{B}{(\log B)^{1/2}}.$$
In particular ``almost all'' plane conics do not contain a rational point.
Serre in fact proved a more general result about counting functions
associated to Brauer group elements, which we now introduce.
Let $X$ be a smooth variety over $F$ equipped with a choice of height function $H$
and let $\br \subset \Br X$ be a finite subset. Let
$$X(F)_\br=\{x \in X(F): b(x)=0 \in \Br F \text{ for all } b \in \br\},$$
denote the ``zero-locus'' of $\br$. Without loss of generality one
may assume that $\br$ is a finite \emph{subgroup}, as
$X(F)_\br = X(F)_{\langle\br\rangle}$.
Define the associated counting function to be
\begin{equation} \label{def:Brauer_counting_function}
	N(X,H,\br,B) = \#\{x \in X(F)_\br:H(x) \leq B\}.
\end{equation}
At first glance it might not be clear how \eqref{def:conic_counting_function}
and \eqref{def:Brauer_counting_function} are related.
However, this relationship becomes immediately clear once one is acquainted with the dictionary
between Brauer group elements and families of Brauer-Severi varieties.
For example, to any conic bundle $\pi:Y \to X$
one may associate a quaternion algebra over $F(X)$, giving rise
to an element $Q_\pi \in \Br U$, where $U \subset X$ is the open subset
given by removing those $x \in X$ whose fibre $\pi^{-1}(x)$ is singular.
If the generic fibre of $\pi$ has the shape
$$ax^2 + by^2 = z^2\quad \subset \PP^2_{F(X)},$$
with $a,b \in F(X)^*$, then the associated quaternion algebra is simply $(a,b)$. Moreover,
given $x \in U(F)$, we have $Q_\pi(x)=0 \in \Br F$ if and only if
the fibre over $x$ contains a rational point. In this case, we therefore
have an equality $N(U,H,Q_\pi,B) = N(U,H,\pi,B)$ of counting functions.
Finite collections of Brauer group elements correspond to \emph{products} of Brauer-Severi
varieties (see \S \ref{sec:Brauer-Severi_schemes} for further details).

Serre \cite{Ser90} only considered the case where
$X=\PP^n,F=\QQ$ and $\br=\{b\}$ is a single element of order two.
For any open subset $U \subset \PP^n$ where $b$ is defined, he showed that
\begin{equation} \label{eqn:Serre}
	N(U,H,b,B) \ll \frac{B}{(\log B)^{\Delta_{\PP^n}(b)}},\quad \text{ as } B \to \infty,
\end{equation}
where $H$ is an anticanonical height function on $\PP^n$ and
$$\Delta_{\PP^n}(b) = \sum_{ D \in (\PP^n)^{(1)}}\left(1 - \frac{1}{|\res_D(b)|}\right).$$
Here, for any variety $X$, we denote by $X^{(1)}$ the set of codimension one points
of $X$.
Also $\res_D(b)$ denotes the residue of $b$ at $D$, which detects whether or not $b$ admits
a singularity along $D$.
In \cite{Ser90}, Serre asked whether the bounds given in (\ref{eqn:Serre}) were sharp.
To the author's knowledge, the corresponding lower bounds
have only been shown in two cases: For the family of
all plane conics over $\QQ$  \cite{Hoo07} and the family
of all plane \emph{diagonal} conics over $\QQ$ \cite{Hoo93, Guo95}.
Moreover no asymptotic formulae have been achieved
for the problem as stated here (there are however results for
related problems on \emph{integral} points, for example classical
work of Landau \cite{Lan08} on the number of integers which may be written as a sum of two squares).

\subsection{Statement of results}
The main result of this paper concerns the counting functions (\ref{def:Brauer_counting_function})
for toric varieties. Recall that an algebraic torus $T$ over a field $k$ is an algebraic group
over $k$ which becomes isomorphic to $\Gm^n$ over an algebraic closure of $k$,
for some $n \in \NN$. We say that such a torus is \emph{anisotropic}
if it has trivial character group over $k$, i.e.~$\Hom(T,\Gm)=0$. In this paper,
a toric variety for $T$ will be a smooth projective variety with a faithful action of $T$ that has an
open dense orbit which contains a rational point. Here Manin's conjecture
is known by work of Batyrev and Tschinkel  \cite{BT98}.
They constructed a special anticanonical height function on toric varieties
which is particularly well-behaved (see \S \ref{sec:Heights}).
In the case of $\Gm^n \subset \PP^n$, this height function is simply the $(n+1)$-th power of usual
height function (\ref{eq:height}).

\begin{theorem} \label{thm:Brauer}
	Let $F$ be a number field and let $T$ be an anisotropic torus over $F$.
	Let $X$ be a toric variety over $F$ with respect to $T$ and
	let $U \subset X$ denote the open dense orbit.
	Let $\br \subset \Br_1 U$ be a finite subgroup and suppose that the zero-locus
    $U(F)_\br$ of $\br$ is non-empty.
	If $H$ denotes the Batyrev-Tschinkel anticanonical height function on $X$,
	then there exists a constant $c_{X,\br,H}>0$ such that
	$$N(U,H,\br,B) \sim c_{X,\br,H} B \frac{(\log B)^{\rho(X)-1}}{(\log B)^{\Delta_X(\br)}}, \quad \text{as } B \to \infty,$$
   	where 
   	$$\Delta_X(\br)=\sum_{D \in X^{(1)}}\left(1 - \frac{1}{|\res_D(\br)|}\right),
   	\qquad \rho(X)=\rank \Pic X.$$
\end{theorem}
Here $\res_D$ denotes the residue map associated to $D$ and $\Br_1 U = \ker(\Br U \to \Br U_{\Fbar})$
denotes the algebraic Brauer group of $U$.
Note that the theorem implies the non-obvious fact that if $U(F)_\br \neq \emptyset$ then $U(F)_\br$ is infinite,
and moreover we shall even show that $U(F)_\br$ is Zariski dense in $U$.
To prove Theorem \ref{thm:Brauer} we first choose an embedding $T \subset X$
to identify $T \cong U$ in such a way that $\br \subset \Br_1 T$ and 
$b(1)=0$ for each $b \in \br$. A result of Sansuc \cite[Lem.~6.9]{San81} states
that as each $b \in \br$ is algebraic, the associated evaluation map $T(F) \to \Br F$ is a group homomorphism.
In particular, the zero-locus $T(F)_\br$ of $\br$ is a subgroup of $T(F)$, hence has a rich structure.
This is one of the main reasons why we focus on algebraic Brauer group elements, as 
Sansuc's result does not hold for transcendental Brauer group elements
and other methods will be required to handle these.
We then proceed by introducing a height zeta function
$$Z(s)=\sum_{\substack{t \in T(F)_\br}} \frac{1}{H(t)^s},$$
in a complex variable $s$. The analytic properties of $Z(s)$ can be related to the original
counting problem via a Tauberian theorem. We study $Z(s)$
using harmonic analysis (in particular Poisson summation) to obtain a continuation
of $Z(s)$ to the line $\re s =1$, away from $s=1$. The fact that $T$
is anisotropic greatly simplifies the Poisson summation
formula, as the quotient $T(\Adele_F) / T(F)$ is compact
in this case.
This harmonic analysis approach is based on the works \cite{BT95}, \cite{BT98} and \cite{CLT14},
which study rational and integral points of bounded height on toric varieties.
A detailed treatment of this approach can be found in \cite{Bou11},
which also handles the function field case.

It would be interesting to remove the anisotropic assumption from
Theorem \ref{thm:Brauer}, however our methods do not give this.
Batyrev and Tschinkel also dealt with the anisotropic case first \cite{BT95},
and to take care of the integral which arises in the Poisson formula in the case
of arbitrary tori, they had to prove a certain technical integration theorem in 
complex analysis \cite[Thm.~6.19]{BT98}. Their Fourier transforms
were meromorphic, so this integration could be achieved using the residue theorem.
This method breaks down in our case, as our Fourier transforms
have \emph{branch point singularities}.
Further technicalities arise from the fact that current zero-free regions for Hecke $L$-functions are
not strong enough to allow us to obtain any kind of continuation of $Z(s)$ to the half plane $\re s <1$
(see Remark \ref{rem:continuation}). 
Therefore, a fundamentally new idea will be required to perform this generalisation to arbitrary tori.


Specialising Theorem \ref{thm:Brauer} to the case $X=\PP^n$,
we provide a positive answer to Serre's question on the sharpness of
(\ref{eqn:Serre}) in some cases
(see \S \ref{sec:application} for explicit examples). Moreover,
not only is it the first result for which an asymptotic formula is achieved,
rather than simply a lower bound, it is also
the first result which applies to varieties other
than $\PP^n$ and to number fields other than $\QQ$.

Our next result is an application of Theorem \ref{thm:Brauer}
to the counting functions \eqref{def:conic_counting_function}
for families of products of Brauer-Severi varieties. 
We also give an interpretation of the exponent $\Delta_X(\br)$ 
appearing in Theorem \ref{thm:Brauer}, purely in terms of the geometry of the
family.
For technical reasons we work with special morphisms which
we call \emph{almost smooth} (see Definition \ref{def:almost_smooth}).
Any dominant morphism between non-singular varieties over a field of characteristic zero
is automatically almost smooth,
so this condition is \emph{weaker} than being non-singular.

\begin{theorem} \label{thm:BS}
	Let $F$ be a number field and let $T$ be an anisotropic torus over $F$.
	Let $X$ be a toric variety over $F$
	with respect to $T$ and let $U \subset X$ denote the open dense orbit.
	Let $Y$ be a variety over $F$ equipped with a proper surjective almost smooth morphism $\pi:Y \to X$
	such that $Y(F) \neq \emptyset$. Suppose that
	\begin{itemize}
		\item $\pi$ admits a rational section over $\Fbar$.
		\item $\pi^{-1}(U) \to U$ is isomorphic to a product of Brauer-Severi
		schemes over $U$.
	\end{itemize}
	For each $D \in X^{(1)}$, choose an irreducible component
	$D'$ of $\pi^{-1}(D)$ of multiplicity one such that $[F(D)_{D'}:F(D)]$
	is minimal amongst all irreducible components of $\pi^{-1}(D)$ of multiplicity one,
	where $F(D)_{D'}$ denotes the algebraic closure of $F(D)$ inside $F(D')$.
	If $H$ denotes the Batyrev-Tschinkel anticanonical height function on $X$,
	then there exists a constant $c_{X,\pi,H}>0$ such that
   	$$N(U,H,\pi,B) \sim c_{X,\pi,H} B \frac{(\log B)^{\rho(X)-1}}{(\log B)^{\Delta_X(\pi)}}, \quad \text{as } B \to \infty,$$
	where 
	$$\Delta_X(\pi)=\sum_{D \in X^{(1)}}\left(1 - \frac{1}{[F(D)_{D'}:F(D)]}\right),
	\qquad \rho(X)=\rank \Pic X.$$
\end{theorem}

We emphasise that Theorem \ref{thm:BS} is many respects just a reformulation of Theorem~\ref{thm:Brauer};
its primary purpose is to give a geometric interpretation of the factor $\Delta_X(\pi)$.
Theorem \ref{thm:BS} gives an answer to Question \ref{Question:3},
as posed at the beginning of the introduction. Moreover, we see that
as soon as there exists some $D'$ in the theorem with $[F(D)_{D'}:F(D)] >1$,
then ``almost all'' of the varieties in the family do not contain a rational point.
As an application, if $\pi:Y \to X$ is a conic bundle which satisfies the conditions of Theorem \ref{thm:BS},
then one can easily show that
$$\Delta_X(\pi)=1/2\cdot \#\{D \in X^{(1)}: \pi^{-1}(D) \text{ is non-split}\}.$$
Here we say that a reduced conic over a perfect field is non-split if it is isomorphic to two intersecting lines which are
conjugate over a quadratic extension. The assumptions of the theorem imply that
there is always an irreducible component of multiplicity one above each point of codimension one,
in particular non-reduced conics do not occur (see Lemma \ref{lem:Delta=Delta}). In general,
only the ``non-split'' fibres contribute to $\Delta_X(\pi)$ (see \S \ref{sec:conclusion}).

We also calculate the leading constant $c_{X,\br,H}$ appearing
in Theorem \ref{thm:Brauer}. This formally resembles the leading constant $c_{X,H,\mathrm{Peyre}}$
conjectured to appear by Peyre \cite{Pey95} in the context of Manin's conjecture.
In Lemma \ref{lem:Peyre} we give examples  where $\Delta_X(\br) = 0$ and $c_{X,\br,H} = c_{X,H,\mathrm{Peyre}}$ yet
$\br \neq 0$, and also examples where again $\Delta_X(\br) = 0$ but  $0<c_{X,\br,H} < c_{X,H,\mathrm{Peyre}}$ 
(such phenomenon does not occur in the case of projective space considered by Serre in \cite{Ser90}).
To assist with the calculation of $c_{X,\br,H}$ we prove
the following result, which should be of independent interest.

\begin{theorem} \label{thm:Brauer_Manin}
	Let $U$ be a principal homogeneous space under an algebraic torus over a number field $F$.
	Let $V \to U$ be a product of Brauer-Severi schemes over $U$
	that admits a rational section over $\Fbar$. Then the Brauer-Manin
	obstruction is the only one to the Hasse principle and weak approximation
	for any smooth proper model of $V$.
\end{theorem}
The Brauer groups of the smooth proper models occurring in Theorem \ref{thm:Brauer_Manin}
are finite (modulo constants). Therefore Theorem \ref{thm:Brauer_Manin} allows one, at least theoretically,
to answer Question \ref{Question:1} for such families. Moreover as soon as there is a rational point on $V$, we see that $V(F)$ is Zariski dense in $V$, which answers
Question \ref{Question:2} (this also shows that $U(F)_\br$ is Zariski dense in $U$, in the notation of Theorem \ref{thm:Brauer}).

As a special case of Theorem \ref{thm:Brauer_Manin}, we obtain a new proof of the following fact:
Let $U \subset \PP^1$ be the complement of two rational points or a closed point of degree two.
Then the Brauer-Manin obstruction is the only one to the Hasse principle and weak approximation
for any smooth proper model $X$ of a product of Brauer-Severi schemes over $U$
(see e.g.~\cite[Thm.~0.4]{Sko96}).
Indeed, here $U$ may be given the structure of an algebraic torus and moreover 
Tsen's theorem implies that $X$ obtains a rational section over $\Fbar$.
Our proof of this fact is different from previous proofs, 
as we do not need to construct explicit models of $X$ over $\PP^1$,
nor do we use descent, the fibration method nor tools from analytic number theory.
We obtain similar results for products of Brauer-Severi schemes defined over certain open subsets
of higher dimensional projective spaces (e.g.~complements of $n+1$ general
Galois conjugate hyperplanes in $\PP^n$), provided there is a rational section
over $\Fbar$. To the author's knowledge, all previous results of this latter type
have been conditional on Schinzel's hypothesis
(e.g.~\cite[Cor.~3.6]{Wit07}).
We prove Theorem \ref{thm:Brauer_Manin} by showing that $V$ itself is stably birational
to a principal homogeneous space under an algebraic torus.
This reduces the problem to the work of Sansuc \cite{San81}, who has already shown that the Brauer-Manin
obstruction is the only one to the Hasse principle and weak approximation in this case.

To also aid with the calculation of the leading constant
in Theorem \ref{thm:Brauer}, we prove a result
(Theorem \ref{thm:CTSD}) which calculates
the unramified Brauer group of a product of Brauer-Severi schemes, generalising
a result of Colliot-Th\'{e}l\`{e}ne and Swinnerton-Dyer \cite[Thm.~2.2.1]{CTSD94}
(which only applies when the base is an open subset of $\PP^1$).
To do this we introduce ``subordinate'' Brauer group elements, which were
also considered by Serre in the case of $\PP^1$ in the appendix of \cite[Ch.~II]{Ser97}.
In \S \ref{Sec:subordinate} we perform a detailed study of the theory of subordinate Brauer group elements, as the author
expects that they shall feature heavily in the analysis of counting functions of the type (\ref{def:Brauer_counting_function}).

\subsection{An application} \label{sec:application}
We now show that Theorem \ref{thm:Brauer} is non-empty
by giving some explicit families of varieties
to which it applies.
Let $n \in \NN$ and let $F \subset E$ be a field extension of
degree $n+1$. The Weil restriction $\Res{E}{F} \Gm$
of $\Gm$ is an algebraic torus over $F$ which may be identified
with the subset $\Norm{E}{F}(t) \neq 0$ in $\mathbb{A}_F^{n+1}$ (see \cite[Ch.~7.6]{BLR90}).  
Hence, we may realise $\PP^n$ as a toric
variety with respect to the anisotropic torus $T_{E/F}=\Res{E}{F}\Gm/\Gm$,
whose boundary is the irreducible divisor $D_{E/F}=\{\Norm{E}{F}(t)=0\}$.
Let $r \in \NN$ and let $F \subset E_i$ be cyclic field extensions
of degree $n_i$ for each $i=1,\ldots,r$, with associated norm forms $\Norm{E_i}{F}$.
Let
\begin{equation} \label{eqn:Z_i}
	Z_i:\quad  \Norm{E_i}{F}(x_1,\ldots,x_{n_i})=\Norm{E}{F}(1,t_1,\ldots,t_n) \subset \mathbb{A}^{n_i} \times \mathbb{A}^{n},
\end{equation}
with associated projections $\pi_i:Z_i \to \mathbb{A}^n$. The smooth fibres of $\pi_i$
are principal homogeneous spaces under the norm one torus
$\smash{\Res{E_i}{F}^1 \Gm}$.
Let $\pi:Z=Z_1 \times_{\mathbb{A}^n} \cdots \times_{\mathbb{A}^n} Z_r \to \mathbb{A}^n$.
We now introduce the Brauer group elements which control the arithmetic of $\pi$.
Let $\chi_i: \Gal(E_i/F) \cong \ZZ/n_i\ZZ$ be a choice of isomorphism and let
$$b_i = (\chi_i, \Norm{E}{F}(1,t_1,\ldots,t_n)),$$
be the associated cyclic algebra over $F(t_1,\ldots,t_n)$ (see \cite[\S2.5]{GS06}).
If $n_i \mid (n+1)$, then $b_i$ is unramified at
$t_0=0$ and hence $b_i \in \Br T_{E/F}$.
Given a point $t \in T_{E/F}(F)$ with $t_0=1$, we have $b_i(t)=0$
if and only if $\Norm{E}{F}(1,t_1,\ldots,t_n)$ is a norm
of some element in $E_i$, i.e.~if and only if the fibre $\smash{\pi_i^{-1}(t)}$
contains a rational point. A similar property holds when $t_0 = 0$,
hence we obtain an equality
$$N(T_{E/F},H,\pi,B) = N(T_{E/F},H,\{b_1,\ldots,b_r\},B),$$
for each height function $H$ on $\PP^n$.
The residue of $b_i$ at $\smash{D_{E/F}}$ is exactly the element of $ \Hom(\Gal(\Fbar/E),\ZZ/n_i\ZZ)$
induced by $\chi_i$ via the maps $\Gal(\Fbar/E) \subset \Gal(\Fbar/F) \to \Gal(E_i/F)$.
Finally as $b_i(1)=0$ and $b_i \otimes_F E_i = 0$ for each $i=1,\ldots,r$,
we see that Theorem \ref{thm:Brauer} applies. For sufficiently general $E_i$,
we obtain the following.

\begin{corollary} \label{cor:example}
	Let $E_i,E,F$ and $\pi$ be as above with $n_i \mid (n+1)$ for each $i=1,\ldots,r$.
	Assume that $E,E_1,\ldots,E_r$ are linearly disjoint over $F$.
	If $H$ denotes the Batyrev-Tschinkel anticanonical height function for $T_{E/F} \subset \PP^n$,
	then there exists a constant $c_{T_{E/F},\pi,H} >0$ such that
	$$N(T_{E/F},H,\pi,B) \sim c_{T_{E/F},\pi,H}\frac{B}{(\log B)^{1 - 1/n_1\cdots n_r}}, 
	 \quad \text{ as } B \to \infty.$$
\end{corollary}
Similar examples are given by replacing the right-hand side of (\ref{eqn:Z_i}) 
with a \emph{product} of norm forms. 
When $r=1$ and $F \subset E_1$ is quadratic,
the equation (\ref{eqn:Z_i}) becomes
$$x_1^2 - ax_2^2 =\Norm{E}{F}(1,t_1,\ldots,t_n),$$
for some $a \in F^*$ which is not a square in $E$. In particular
Corollary \ref{cor:example} provides a positive answer to Serre's question
in such cases when $2 \mid n+1$.

\subsection{Conjectures} \label{sec:conclusion}
We now speculate about generalisations of 
Theorem \ref{thm:Brauer} and Theorem \ref{thm:BS}. In what follows we only 
consider the following varieties.
\begin{assumption}
	$X$ is a smooth projective variety over a number field $F$ such that:
	\begin{enumerate}
		\item The anticanonical bundle of $X$ is big, $\Pic X$ is torsion free and  \label{asm:cohom}
		$$\HH^1(X,\OO_X) = \HH^2(X,\OO_X)=0.$$
		\item There exists $U \subset X$ open such that \label{item:equi}
		for every anticanonical height function $H$ on $X$ we have
		$$N(U,H,B) \sim c_{X,H,\mathrm{Peyre}}B(\log B)^{\rho(X) - 1}.$$
	\end{enumerate}
\end{assumption}
For example, Fano varieties and equivariant compactifications of connected linear algebraic groups
and their homogeneous spaces \cite[Thm.~1.2]{FZ13} satisfy Condition (\ref{asm:cohom}).
Condition \eqref{item:equi} says that the rational points on $X$ are equidistributed (see \cite[\S3]{Pey95}).
This property is known to hold for toric varieties \cite[Thm.~3.10.3]{CLT14}, for example.

\subsubsection{Zero-loci of Brauer group elements}
For zero-loci of Brauer group elements, our conjecture is as follows.

\begin{conjecture} \label{conj:Serre}
	Let $X$ satisfy the above assumptions and let $H$ be an
	anticanonical height function on $X$.
	Let $\br \subset \Br F(X)$ be a finite subgroup and suppose that there exists $x \in X(F)$ such that
	all $b \in \br$ are defined at $x$ and $b(x)=0$.

	Then there exists $U \subset X$ open with $\br \subset \Br U$ and $c_{U,\br,H} > 0$	such that
	$$N(U,H,\br,B) \sim c_{U,\br,H}B \frac{(\log B)^{\rho(X)-1}}{(\log B)^{\Delta_X(\br)}}, \quad \text{as } B \to \infty,$$
	where
	$$\Delta_X(\br)=\sum_{D \in X^{(1)}}\left(1 - \frac{1}{|\res_D(\br)|}\right).$$
\end{conjecture}

Note that in the case where $X = \PP^n$, we conjecture a positive answer
to the question posed by Serre in \cite{Ser90}. One can prove this
conjecture in some simple cases.

\begin{theorem} \label{thm:trivial}
	Conjecture \ref{conj:Serre} holds if $\br \subset \Br X$.
\end{theorem}
\begin{proof}
	As $\br \subset \Br X$, there exists a finite set of places
	$S$ of $F$ such that $X(F_v)_\br = X(F_v)$ for all $v \not \in S$.
	Moreover, $\Delta_X(\br)=0$ and $X(F_v)_\br \subset X(F_v)$ is open and closed
	for all $v \in S$.
	Our assumptions on $X$ imply that its rational points are equidistributed.
	As we are simply counting those rational points on $X$ 
	which satisfy a finite list of well-behaved $v$-adic conditions,
	the asymptotic formula follows from \cite[Prop.~3.3]{Pey95}.
\end{proof}

The language of virtual Artin representations (see \S \ref{sec:virtual_Artin}) provides a useful formalism
to describe the factors appearing in this conjecture. Namely in the setting of Conjecture \ref{conj:Serre},
the virtual Artin representation of interest is
$$\Pic_\br(\overline{X})_\CC=\Pic(\overline{X})_\CC
- \sum_{D \in X^{(1)}}\left(1-\frac{1}{|\res_{D}(\br)|}\right)\Ind_{F_D}^F \CC,$$
where $F_D$ denotes the algebraic closure of $F$ inside $F(D)$
(note that $\rank \Pic_\br(\overline{X})_\CC^{G_F} = \rho(X) - \Delta_X(\br)$). 
We also have a good idea of what shape the leading constant in Conjecture \ref{conj:Serre} should take
(see \S \ref{sec:leading_constant}).
It should involve the size of a group of subordinate Brauer group elements,
which in special cases can be related to the size of the unramified Brauer group of the associated
product of Brauer-Severi varieties. In \S \ref{sec:leading_constant} we also construct
a Tamagawa measure, whose convergence factors are provided by the local factors of the 
virtual Artin $L$-function $L(\Pic_\br(\overline{X})_\CC,s)$.
Our proof that these are indeed a family of convergence factors in our case hinges on the calculation of
the local Fourier transforms of the height functions.
It would be interesting if one could prove that these provide a family of convergence factors 
in general.

\subsubsection{Examples}
We now give some simple examples for which Conjecture \ref{conj:Serre}
is not known. Let $F$ be a number field and let $H$ be an anticanonical height function
on $\PP^1$. Let $f_0,f_1,f_2 \in F[t]$ have degrees $d_0,d_1,d_2$ and let $f=f_0f_1f_2$.
Consider the variety
$$Y: f_0(t)x_0^2  + f_1(t)x_1^2 + f_2(t)x_2^2 = 0 \quad \subset \mathbb{A}^1 \times \PP^2.$$
This corresponds to the class $b$ of the quaternion algebra $(-f_1(t)/f_0(t),-f_2(t)/f_0(t))$ in $\Br F(t)$.
The natural projection $\pi:Y \to \mathbb{A}^1$ realises $Y$ as a conic bundle over $\mathbb{A}^1$.
We may modify any finite collection of fibres of $\pi$ without changing the asymptotic formula
obtained, in particular we may assume that $f$ is separable. Also by contracting Galois orbits of
disjoint $(-1)$-curves, we may assume that $-f_j(\alpha_i)/f_k(\alpha_i)$ is not a square in
the field $F(\alpha_i)$, for each root $\alpha_i$ of $f_i$ over $\Fbar$ and $\{i,j,k\} = \{0,1,2\}$.
We also assume that $d_0 \equiv d_1 \equiv d_2 \mod 2$
(this corresponds to asking that the fibre at infinity be smooth). Under these assumptions, 
if $Y(F) \neq \emptyset$ then Conjecture \ref{conj:Serre} predicts 
\begin{equation} \label{eqn:DP}
	N(\mathbb{A}^1,H,\pi,B) \sim \frac{c_{\mathbb{A}^1,\pi,H}B}{(\log B)^{m/2}}, \quad \text{as } B \to \infty,
\end{equation}
where $m$ is the number of irreducible factors of $f$ over $F$.
If $(d_0,d_1,d_2)=(2,0,0)$ there are two cases. If $f$ is irreducible, then (\ref{eqn:DP})
follows from Corollary \ref{cor:example}, whereas if $f$ is reducible then this follows
from \cite[Thm.~1.3]{BN15}.
The next simplest cases where $(d_0,d_1,d_2) = (1,1,1), (0,2,2)$ or $(4,0,0)$
are wide open (in this last case one obtains a Ch\^{a}telet surface).

\subsubsection{Families of varieties}
We now consider generalisations of Theorem \ref{thm:BS}
to the case where the generic fibre is no longer a product of Brauer-Severi varieties.
Suppose that $X$ satisfies the above assumptions and let $H$ be an anticanonical
height function on $X$. Let $\pi:Y \to X$
be a proper surjective almost smooth morphism with geometrically integral generic fibre 
(see Definition \ref{def:almost_smooth} for the definition of \emph{almost smooth}).
At the moment it seems quite naive to hope for a conjectural framework for the counting
functions \eqref{def:conic_counting_function} in general, due to
complications arising from the possible failure of the Hasse
principle (see for example \cite{Bha14} for some recent spectacular work in a special case). 
More realistic is to study the counting functions
\begin{equation} \label{def:local_counting_function}
    N_{\mathrm{loc}}(U,H,\pi,B) = \#\{x \in U(F): H(x) \leq B, x \in \pi(Y(\Adele_F))\}
\end{equation}
for $U \subset X$ open, which count the number of varieties in the family which are everywhere locally soluble.
Here our investigations suggest the following philosophy.
\begin{philosophy}
	For sufficiently small $U$, the asymptotic behaviour of $N_{\mathrm{loc}}(U,H,\pi,B)$ is controlled by those
	fibres over codimension one points $D \in X^{(1)}$ for which $\pi^{-1}(D)$ is \emph{non-split}.
\end{philosophy}
Here, following Skorobogatov \cite[Def.~0.1]{Sko96},
we say that a scheme over a perfect field $k$ is \emph{split} if it contains
a geometrically integral open subscheme.
We now offer conditions which guarantee that a positive proportion
of the varieties in the family are everywhere locally soluble.
\begin{conjecture} \label{conj:loc}
	Let $\pi:Y \to X$ be as above with $Y(\Adele_F) \neq \emptyset$.
	Assume that $\pi^{-1}(D)$ is split for all $D \in X^{(1)}$ and that $\pi(	Y(\Adele_F)) \cap X(F) \neq \emptyset$.
	Then there exists $U \subset X$ open such that
	$$\lim_{B \to \infty} \frac{N_{\mathrm{loc}}(U,H,\pi,B)}{N(U,H,B)} > 0.$$
\end{conjecture}
We also naturally conjecture that the limit in Conjecture \ref{conj:loc} exists.
Conjecture \ref{conj:loc} holds, for example, if $Y \to X$ is a product of
Brauer-Severi schemes (this follows from Theorem \ref{thm:trivial}).
Its validity in the special case $X = \PP^n$ is a recent theorem of the author,
in joint work with Bright and Browning \cite[Thm.~1.3]{BBL15}.
Unfortunately the converse of Conjecture \ref{conj:loc} does not hold; consider the family of varieties
\begin{equation} \label{eqn:multinorm}
	(x_1^2 - ay_1^2)(x_2^2 - by_2^2)(x_3^2 - aby_3^2) = t \quad \subset \mathbb{A}^6 \times \mathbb{A}^1,
\end{equation}	
for $a,b \in F^*$, viewed as a fibration over $\mathbb{A}^1$. If $a,b,ab \not \in F^{*2}$,
then the fibres over $t=0$ and $t=\infty$ are non-split, however Colliot-Th\'{e}l\`{e}ne \cite[Prop.~5.1(a)]{CT14} has shown
that \emph{every} smooth member of this family is everywhere locally soluble.

Of course there are many interesting families of varieties for which $\pi^{-1}(D)$
can be non-split for some $D \in X^{(1)}$ (families of Brauer-Severi varieties,
quadric bundles of relative dimension two \cite{Sko90}, multinorm equations, elliptic fibrations,\ldots),
and one would also like a conjectural framework in this case. However here
much subtleties arise as \eqref{eqn:multinorm} illustrates. For simplicity
we therefore assume that the fibre $\pi^{-1}(D)$ over each $D \in X^{(1)}$
is integral and that the algebraic closure $F(D)_\pi$ of $F(D)$
in $F(\pi^{-1}(D))$ is Galois.
Our investigations suggest that when $\pi(Y(\Adele_F)) \cap X(F) \neq \emptyset$, there exists an open subset $U \subset X$ and 
$c_{U,\pi,H}>0$ such that
\begin{equation} \label{conj:non-split}
	N_{\mathrm{loc}}(U,H,\pi,B) \sim c_{U,\pi,H} B \frac{(\log B)^{\rho(X)-1}}{(\log B)^{\Delta_X(\pi)}}
\end{equation}
as $B \to \infty$, where 
\begin{equation} \label{eqn:Delta}
	\Delta_X(\pi)=\sum_{D \in X^{(1)}}\left(1 - \frac{1}{[F(D)_{\pi}:F(D)]}\right).
\end{equation}
If the generic fibre of $\pi$ is a product of Brauer-Severi varieties,
then \eqref{conj:non-split} is compatible with
Conjecture \ref{conj:Serre}
(this follows from the results of \S\ref{sec:thm:BS}).

If $F(D)_\pi / F(D)$ is not Galois or $\pi^{-1}(D)$ is not integral, 
then $[F(D)_{\pi}:F(D)]$ in \eqref{eqn:Delta}
should be replaced by the density of an appropriate Frobenian
set of primes.
The example \eqref{eqn:multinorm} is compatible 
with this philosophy; indeed, almost all places of $F$ are split in at least one
of $F(\sqrt{a}),F(\sqrt{b})$ or $F(\sqrt{ab})$ (these being the algebraic closures
of $F$ in the irreducible components of the non-split fibres of \eqref{eqn:multinorm}).
A special case of such a result was also recently obtained in \cite[Thm.~1.3]{BN15}.
Further examples involving Frobenian sets are investigated in great generality in 
forthcoming joint work \cite{LS15} of the author and Smeets.

\subsection*{Outline of the paper:}
In \S \ref{sec:Brauer} we begin by recalling various facts on Brauer groups, Brauer-Severi
varieties and the Brauer-Manin obstruction. We then prove Theorem \ref{thm:BS} (assuming
Theorem \ref{thm:Brauer}) followed by Theorem \ref{thm:Brauer_Manin}. Next we introduce the notion of
subordinate Brauer group elements and study some of their basic properties.
Our main theorem on subordinate Brauer group elements 
concerns their use in calculating unramified Brauer groups of products of Brauer-Severi schemes,
which generalises a result of Colliot-Th\'{e}l\`{e}ne and Swinnerton-Dyer \cite[Thm.~2.2.1]{CTSD94}. We also prove an analogue
for subordinate Brauer group elements of a result of Harari \cite[Thm.~2.1.1]{Har94}, 
which makes clear their relevance in the study of zero-loci of Brauer group elements.

In \S \ref{sec:L-functions} we gather various results on Hecke $L$-functions and certain
partial Euler products which shall naturally appear in the proof of Theorem \ref{thm:Brauer}. We also state here
the version of Delange's Tauberian theorem which we shall use.

In \S \ref{sec:Tori} we begin with some results on the adelic spaces and Brauer groups of algebraic tori,
in particular we calculate their algebraic Brauer group
in terms of automorphic characters.
We then study toric varieties and their Brauer groups over number fields,
and derive an analogue for subordinate Brauer group elements of an exact sequence of Voskresenski\v{\i} \cite{Vos70}
on the unramified Brauer groups of algebraic tori.

In \S \ref{sec:counting} we bring together all the work in the previous sections to prove Theorem \ref{thm:Brauer},
and also give a description of the leading constant.

\subsection*{Acknowledgments:}
Part of this work was completed during the program
``Arithmetic and Geometry'' in 2013 at HIM, Bonn. The author would like to thank
David Bourqui, Martin Bright, Tim Browning, Ariyan Javanpeykar, Arne Smeets and Yuri Tschinkel
for useful comments, and the anonymous referees for their careful reading of our paper.
Special thanks go to Jean-Louis Colliot-Th\'{e}l\`{e}ne,
for many useful conversations and for his help with Theorem \ref{thm:Brauer_Manin}
and Theorem \ref{thm:CTSD}. Thanks also go to Alexei Skorobogatov for his help with the proof of Lemma \ref{lem:dvr_ind},
and to David Harari for pointing out a mistake in an earlier version of Theorem \ref{thm:Brauer_global}.

\subsection*{Notation}
\subsubsection*{Algebra}
Given a ring $R$, we denote by $R^*$ the group of units of $R$.
For a subset $A \subset G$ of a group $G$ we denote by $\langle A \rangle$ the subgroup generated by $A$.
For any element $g \in G$ of finite order, we denote by $|g|$ its order.
Given a topological group $G$, we denote by $\Fourier{G} = \Hom(G,S^1)$ the group of continuous characters of $G$
and $G^\sim = \Hom(G,\QQ/\ZZ)$ the continuous $\QQ/\ZZ$-dual of $G$.
We fix an embedding $\QQ/\ZZ \subset S^1$ so that $G^\sim \subset \Fourier{G}$.
We denote by $\widehat{G}$ the completion of $G$ with respect to the normal open subgroups of finite index.
We make frequent use of the following version of character orthogonality:
Let $G$ be a compact Hausdorff topological group with a Haar measure $\mathrm{d}g$.
Then for any character $\chi$ of $G$ we have
$$\int_G \chi(g) \mathrm{d}g = \left\{
	\begin{array}{ll}
		\vol(G), &\chi=1, \\
		0, & \text{ otherwise}.
	\end{array} \right.$$

\subsubsection*{Number theory}
We denote by $\OO_{F}$ the ring of integers of a number field $F$
and by $\Val(F)$ its set of valuations. For any $v \in \Val(F)$,
we denote by $F_v$ the completion of $F$ at $v$ and by $\OO_v$ its maximal
compact subgroup. Given a non-archimedean place $v$ of $F$, we denote by
$q_v$ the size of the residue field of $F_v$ and $\pi_v$ a choice of uniformiser.
We choose absolute values on each $F_v$ such that $|x|_v=|N_{F_v/\QQ_p}(x)|_p$,
where $v|p \in \Val(\QQ)$ and $|\cdot|_p$ is the usual absolute value on $\QQ_p$.
We normalise our Haar measures $\mathrm{d}x_v$ on each $F_v$ \`{a} la Tate \cite[Ch.~XV]{CF10}.
We have $\vol(\OO_{v}) = 1$
for almost all $v$ and $\vol(\Adele_F/F)=1$ with respect to the associated Haar
measure on the adele group $\Adele_F$ of $F$.

\subsubsection*{Geometry}
For a field $k$, we denote by $\PP_k^n$ and $\mathbb{A}_k^n$ projective $n$-space and affine $n$-space over $k$ respectively.
We sometimes omit the subscript $k$ if the field is clear.  A variety over $k$ is a separated geometrically integral scheme of finite type over $k$.
For every perfect field $k$, we fix a choice of algebraic closure $\kbar$ and we denote by $G_k$ the absolute Galois group of $k$
with respect to $\kbar$. Given a variety $X$ over $k$, we denote $\Xbar=X \times_k \kbar$,
and if $k$ is a number field and $v$ is a place of $k$ then we set $X_v = X\times_k k_v$.
All cohomology will be taken with respect to the \'{e}tale topology.

\section{Brauer groups and Brauer-Severi schemes} \label{sec:Brauer}
In this section we collect various facts on Brauer groups of algebraic varieties and the Brauer-Manin
obstruction, before proving Theorem \ref{thm:BS} and Theorem \ref{thm:Brauer_Manin}.
We then move onto a detailed study of the theory of subordinate Brauer group elements.

\subsection{Brauer groups of varieties}
We now recall some standard facts on Brauer groups,
as  can be found in \cite{Gro68}, \cite[\S1]{CTSD94} and \cite{GS06}.
Let $X$ be a smooth variety over a field $k$ of characteristic zero,
with Brauer group $\Br X = \HH^2(X,\Gm)$.
We let $\Br_0 X=\im(\Br k \to \Br X)$,
$\Br_1 X = \ker(\Br X \to \Br \Xbar)$ 
and $\Br_a X = \Br_1 X / \Br_0 X$.

\subsubsection{The Brauer-pairing}
Fundamental to this paper is the pairing
\begin{align} \label{def:Brauer_pairing}
	\Br X \times X(k)  \to \Br k, \qquad
	(b,x) \mapsto b(x).
\end{align}
Here $b(x)$ denotes the evaluation of $b$ at $x$, namely the pull-back of $b$ via the morphism $\Spec k \to X$
associated to the rational point $x$. This pairing is additive on the left and functorial.
Given a subset $\mathscr{B} \subset \Br X$, we denote its ``zero-locus'' by
\begin{equation} \label{eqn:Brauer_kernel}
	X(k)_{\br}=\{x \in X(k):b(x)=0 \text{ for all } b \in \br\},
\end{equation}
 We  denote by $\thorn_\br$ the indicator function of the set $X(k)_\br$
(or simply $\thorn$ if $\br$ is clear).

\subsubsection{Residues}
For any discrete valuation $v$ on the function field $k(X)$ of $X$
there is a residue map
$$\res_v: \Br k(X) \to \HH^1(k(v),\QQ/\ZZ),$$
where $k(v)$ denotes the residue field of $v$. If $R \subset k(X)$ is the associated
discrete valuation ring $R$, we shall  sometimes denote this by $\res_R$.
If $D \in X^{(1)}$, we shall also denote by $\res_D$
the residue map associated to the corresponding discrete valuation. 

\subsubsection{Purity and the unramified Brauer group}
If $U \subset X$ is a dense open subset, then by Grothendieck's purity theorem the residue maps give
rise to exact sequences 
\begin{equation} \label{seq:purity0}
	0 \to \Br X \to \Br U \to \smashoperator{\bigoplus_{D \in X^{(1)} \setminus U^{(1)}}} \HH^1(k(D),\QQ/\ZZ),
\end{equation}
and
\begin{equation} \label{seq:purity}
	0 \to \Br_1 X \to \Br_1 U \to \smashoperator{\bigoplus_{D \in X^{(1)} \setminus U^{(1)}}} \HH^1(k_D,\QQ/\ZZ),
\end{equation}
where $k_D=\kbar \cap k(D) \subset \kbar(D)$ (see \cite[Lem.~14]{CTS77} for this last sequence).
We also have
$$\Br X = \{b \in \Br k(X) : \res_v(b) = 0 \text{ for all } v \in X^{(1)}\}.$$
The unramified Brauer group $\Brnr(k(X)/k)$ is defined to be the group formed by those $b \in \Br k(X)$
for which $\res_v(b)=0$ for all discrete valuations
$v$ of $k(X)$ which are trivial on $k$. We write this as $\Brnr X$ if $k$ is clear.
It is isomorphic to the Brauer group of any smooth proper
model of $X$ (see \cite[Thm.~1.3.2]{CTSD94}).

\subsection{Brauer groups over number fields} \label{sec:Br_fields}
Let $F$ be a number field. 
The sum of the invariant maps $\inv_v$ gives rise to the following fundamental
short exact sequence
\begin{equation} \label{eqn:CFT}
	0 \to \Br F \to \bigoplus_{v \in \Val(F)} \Br F_v \to \QQ/\ZZ \to 0
\end{equation}
from class field theory.
If $X$ is a smooth variety over $F$, we let
\begin{equation} \label{def:Be}
	\Be(X)=\ker\left(\Br_1 X \to \prod_{v \in \Val(F)} \Br_1 X_{v}\right).
\end{equation}
For all $v \in \Val(F)$ the local pairings	
$$\Br X_v \times X(F_v) \to \QQ/\ZZ,$$
are locally constant on the right. Moreover, for $b \in \Br X$ and a model $\mathcal{X}$
for $X$ over $\OO_F$, the induced map $\mathcal{X}(\OO_v) \to \QQ/\ZZ$ is trivial for
almost all $v$ (i.e.~only takes the value $0$ for almost all $v$).
In particular, the sum of the local pairings gives rise to well-defined pairings
$$\Br X \times X(\Adele_F) \to \QQ/\ZZ, \quad \Brnr X \times \smashoperator{\prod_{v \in \Val(F)}}X(F_v) \to \QQ/\ZZ,$$
which are locally constant on the right and trivial on $X(F)$.
For any $\br \subset \Br X$ we let
$$X(\Adele_F)_\br=\{(x_v) \in X(\Adele_F):b(x_v)=0 \text{ for all } v \in \Val(F) \text{ and all } b \in \br\},$$
denote the ``adelic zero-locus'' of $\br$. The left-exactness of (\ref{eqn:CFT}) implies that
$X(F)_\br = X(\Adele_F)_\br \cap X(F)$.
Moreover
$$
	 X(\Adele_F) \to \{0,1\}, \qquad
	(x_v) \mapsto \prod_{v \in \Val(F)} \thorn_{\br_v}(x_v),
$$
is a continuous extension of $\thorn_\br$ to $X(\Adele_F)$, which by abuse of notation we shall also denote by $\thorn_\br$.
For any subset $Z \subset X(\Adele_F)$, we shall write $Z_\br =Z \cap X(\Adele_F)_\br$.

\subsubsection{The Brauer-Manin obstruction}
We now recall some facts about the Brauer-Manin obstruction as can
be found in \cite[\S5]{Sko01}. We denote by
$\overline{X(F)}$ (resp.~$\smash{\overline{X(F)}^w}$) the closure of $X(F)$ in $X(\Adele_F)$ 
with respect to the adelic (resp.~product) topology.
Given a subset $\br \subset \Br X$ we shall denote by
$$X(\Adele_F)^\br = \{ (x_v) \in X(\Adele_F): \sum_{v \in \Val(F)} \inv_v b(x_v)=0 \text{ for all } b \in \br\}.$$
We shall combine this notation with the notation for zero-loci of Brauer group elements.
Namely for any $\br_1,\br_2 \subset \Br X$
and any subset $Z \subset X(\Adele_F)$ the notation $\smash{Z^{\br_2}_{\br_1}}$ means
$$Z^{\br_2}_{\br_1} = \left\{ (x_v) \in Z:
\begin{array}{l}
	b_1(x_v)=0 \text{ for all } b_1 \in \br_1\text{ and all } v \in \Val(F),\\
	\sum_{v \in \Val(F)} \inv_v b_2(x_v)=0 \text{ for all } b_2 \in \br_2
\end{array}
\right\}.$$
If $\br \subset \Brnr X$, we denote by
$$\left(\prod_{v \in \Val(F)}X(F_v)\right)^\br = \{ (x_v) \in \prod_{v \in \Val(F)}X(F_v): \sum_{v \in \Val(F)} \inv_v b(x_v)=0 \text{ for all } b \in \br\}.$$
We say that the Brauer-Manin obstruction is the only one to the Hasse principle
and weak approximation for $X$ if $X(F)$ is dense in $(\smash{\prod_{v \in \Val(F)}} X(F_v))^{\Brnr X}$.
When $\Brnr X/ \Br_0 X$ is finite,
the Brauer-Manin obstruction is the only one to the Hasse principle
and weak approximation for $X$ if and only if
$\overline{X(F)} = X(\Adele_F)^{\Brnr X}.$

\subsection{Brauer-Severi schemes and their products} \label{sec:Brauer-Severi_schemes}
Let $X$ be a smooth  variety over a field $k$ of characteristic zero.
A Brauer-Severi scheme of relative dimension $n$ over $X$ is a scheme $Y$ over $X$ which is \'{e}tale
locally on $X$ isomorphic to $\PP_X^n$. This corresponds to an element of
$\HH^1(X,\PGL_{n+1})$. Consider the exact sequence
\begin{equation} \label{eqn:Brauer-Severi}
\HH^1(X,\GL_{n+1}) \to \HH^1(X,\PGL_{n+1}) \stackrel{}{\to} \Br X,
\end{equation}
of pointed sets. Denote by $[Y]$ the corresponding
element of $\Br X$. If $X$ is quasi-projective, then a 
theorem of Gabber \cite{Gab81} states that every
element of $\Br X$ is of the form $[Y]$ for some Brauer-Severi scheme $Y$ over $X$.

The exactness of (\ref{eqn:Brauer-Severi}) implies that $[Y]=0$ if and only if
$Y$ is \emph{Zariski} locally on $X$  isomorphic to $\PP^n_X$.
In particular, a Brauer-Severi variety over $k$
has trivial class in $\Br k$ if and only if it is isomorphic to $\PP^n_k$,
which occurs if
and only if it has a rational point. Hence the inclusion $\Br X \subset \Br k(X)$ 
implies that $[Y]=0$ if and only if the morphism $Y \to X$ has a rational section.
In particular $[Y] \in \Br_1 X$ if and only if $Y \to X$ admits a rational section over $\kbar$.
Similar remarks apply to products of Brauer-Severi schemes. For example, if 
$Y=Y_1 \times_X \cdots \times_X Y_r$ is a product of Brauer-Severi schemes over $X$, then $Y \to X$
has a rational section if and only if $[Y_i]=0$ for each $i=1,\ldots,r$.

\subsection{Proof of Theorem \ref{thm:BS}} \label{sec:thm:BS}
We now begin the proof of Theorem \ref{thm:BS} (assuming Theorem \ref{thm:Brauer}).
Let $\pi:Y \to X$ and $U \subset X$ be as in Theorem \ref{thm:BS}. Let $\br \subset \Br U$
denote the subgroup corresponding to the product of Brauer-Severi schemes given by $\pi^{-1}(U) \to U$.
The remarks in the previous section immediately imply that
$$N(U,H,\br,B) =N(U,H,\pi,B),$$
in the notation of Theorem \ref{thm:Brauer} and Theorem \ref{thm:BS}.
Therefore to prove the result, it suffices to show that $\Delta_X(\br)=\Delta_X(\pi)$.
To do this we may work locally around each divisor,
so we begin with some results on schemes over discrete valuation rings.
Throughout this section we let $R$ be a discrete valuation ring with field of fractions $K$
and perfect residue field $k$. We shall work with the following types
of schemes.

\begin{definition} \label{def:almost_smooth}
	Let $X$ be a finite type separated scheme over $R$. We say that $X$ is \emph{almost smooth}
	over $R$ if its generic fibre is smooth and the smooth locus $X^{\mathrm{sm}}$ 
	of $X$ over $R$ is a weak N\'{e}ron model \cite[Def.~3.5.1]{BLR90}
	for $X$, i.e.,~for any \'{e}tale $R$-algebra $R'$, each $R'$-point of
	$X$ lies in the smooth locus $X^{\mathrm{sm}}$ of $X$. 
	
	We say that a scheme $X$ over a normal
	scheme $Y$ is \emph{almost smooth} if it is almost smooth over the local ring of every codimension one point of $Y$.
\end{definition}
Examples of almost smooth schemes include smooth schemes  and regular schemes (see \cite[p.~61]{BLR90}).
The proof of the following result is based on the proof of \cite[Lem.~3.8]{Wit07} (see also \cite[Lem.~1.1]{Sko96}).
For an integral scheme $X$ of finite type over $k$, we denote by $k_X$ the algebraic closure of $k$ in $k(X)$,
which is a finite field extension of $k$.
\begin{lemma} \label{lem:dvr_ind}
	Let $\pi_i : X_i \to \Spec R$ be flat proper almost smooth integral schemes over $R$, for $i=1,2$.
	Assume that the generic fibres of $\pi_1$ and $\pi_2$ are isomorphic.
	Then for each irreducible component $D_1$ of multiplicity one of the special fibre of $\pi_1$ there exists
	an irreducible component $D_2$ of multiplicity one of the special fibre of $\pi_2$ such that
	$$k_{D_2} \subset k_{D_1}.$$
	Choose such a component $D_1$ (if it exists) for which $[k : k_{D_1}]$ is minimal amongst
	the irreducible components of multiplicity one of the special fibre of $\pi_1$.
	Then there exists an irreducible component $D_2$ of multiplicity one of the special fibre of $\pi_2$ such that
	$$k_{D_1} \cong k_{D_2}, \quad \mbox{as $k$-algebras.}$$	
\end{lemma}
\begin{proof}
	Let $d_1$ be the generic point of $D_1$ and let $R_1$ be the local ring 
	at $d_1$; this is a discrete valuation ring by flatness.
	Let $K_1=k(X_1)$ be the field of fractions of $R_1$.
	The inclusion  $\Spec K_1 \hookrightarrow X_1$
	yields a $K_1$-point of $X_1$, hence a $K_1$-point of $X_2$
	which extends to a unique morphism $f_1:\Spec R_1 \to X_2$
	by properness. Let $c_2=f_1(d_1)$.
	Then the residue field $k(c_2)$ of $c_2$ embeds $k$-linearly
	in the residue field $k(d_1)$ of $d_1$. Moreover since
	$R \subset R_1$ has ramification index one and $\pi_1$ is almost smooth,
	we see that $c_2$ lies in the smooth locus of $\pi_2$ (see \cite[Lem.~3.6.5]{BLR90}).
	Hence $c_2$ lies in a unique irreducible component $D_2$ of multiplicity one of the special fibre of $\pi_2$.
	Let $d_2$ denote the generic point of $D_2$.
	Note that the local rings of $d_2$ and $c_2$ have the same field of fractions
	and moreover the local ring at $c_2$ is integrally closed as it is a regular local ring.
	It follows that $k_{D_2}$ embeds $k$-linearly into the local ring of $c_2$
	and we obtain a $k$-linear embedding $k_{D_2} \hookrightarrow k(c_2)$.
	We have constructed a sequence of inclusions
	$$k \subset k_{D_2} \subset k(c_2) \subset k(d_1).$$
	As $k \subset k_{D_2}$ has finite degree, we see that it lies inside the algebraic
	closure $k_{D_1}$ of $k$ inside $k(d_1) = k(D_1)$, which proves the first part of the lemma.
	To prove the second part of the lemma, choose some $D_2$
	such that $k_{D_2} \subset k_{D_1}$.
	Choose also an irreducible component $D_1'$ of multiplicity one of the special fibre of $\pi_1$ such that $k_{D_1'} \subset k_{D_2}$.
	As $[k_{D_1'}:k] \geq [k_{D_1}:k]$ by assumption, we find that $k_{D_1'} = k_{D_1}$,
	and the result follows.
\end{proof}

The next result, which is a simple application of the work
of Artin \cite{Art82} and Frossard \cite{Fro97}, shows the existence of particularly 
well-behaved almost smooth models for products of Brauer-Severi schemes.

\begin{lemma} \label{lem:AF}
	Let $V = V_1 \times_K \cdots \times_K V_r$ be a product of Brauer-Severi varieties over $K$.
	Then there exists a flat proper almost smooth integral scheme $\mathcal{V} \to \Spec R$
	whose generic fibre is isomorphic to $V$ and whose special fibre is reduced, such that the
	algebraic closure of $k$ in the function field of each irreducible component of the special fibre
	is the compositum of the cyclic field extensions determined by the residues $\res_R([V_1]),\ldots,\res_R([V_r])$.
\end{lemma}
\begin{proof}
	Artin \cite[Thm.~1.4]{Art82} has constructed
	regular flat proper integral schemes $\mathcal{V}_i \to \Spec R$
	whose generic fibres are isomorphic to $V_i$ and whose special fibres are integral, for each $i=1,\ldots,r$.
	Frossard \cite[Prop.~2.3]{Fro97} has shown that the algebraic closure of $k$
	inside the function field of the special fibre of $\mathcal{V}_i$ is exactly the cyclic field extension
	of $k$ determined by the residue $\res_R([V_i])$.
	We now take $\mathcal{V}=\mathcal{V}_1 \times_R \cdots \times_R \mathcal{V}_r$.
	This is obviously flat and proper over $R$ and integral. Moreover, it is easy
	to see that it is almost smooth and that the special fibre is reduced as $k$ is perfect.
	The result therefore follows from the fact that the tensor product of Galois field extensions
	is isomorphic to a direct power of the compositum of those field extensions.
\end{proof}
One cannot in general use the fact that the $\mathcal{V}_i$ are regular
to deduce that $\mathcal{V}$ itself is regular (regularity can fail to hold even for
the fibre product of a conic bundle with itself). It is for this reason that we have introduced
the notion of ``almost smooth'', as it allows us to avoid the need to construct an explicit desingularisation
of $\mathcal{V}$.

\begin{lemma} \label{lem:Delta=Delta}
	Let $X$ be a smooth variety over a field $k$
	of characteristic zero. Let $Y$ be an integral scheme together
	with a proper surjective almost smooth morphism $\pi:Y \to X$.
	Suppose that the generic fibre of $\pi$ is isomorphic to a product of Brauer-Severi varieties $V_1,\ldots,V_r$
	over $k(X)$. Then the fibre over each point of codimension one contains an irreducible component
	of multiplicity one.
	
	For each $D \in X^{(1)}$, choose an irreducible component
	$D'$ of $\pi^{-1}(D)$ of multiplicity one such that $[k(D)_{D'}:k(D)]$
	is minimal amongst all irreducible components of $\pi^{-1}(D)$ of multiplicity one.
	Then for any $D \in X^{(1)}$	we have
	$$[k(D)_{D'}:k(D)] = |\langle \res_D(\br) \rangle|,$$
	where $\br = \{[V_1],\ldots,[V_r]\} \subset \Br k(X)$.
\end{lemma}
\begin{proof}
	To prove the result, we may work locally near each point $D \in X^{(1)}$.
    The result therefore follows from applying Lemma \ref{lem:dvr_ind}
    to the models given in Lemma \ref{lem:AF}.
\end{proof}
\noindent
This shows that $\Delta_X(\pi) = \Delta_X(\br)$, hence completes the proof of Theorem \ref{thm:BS}. \qed

\subsection{Proof of Theorem \ref{thm:Brauer_Manin}} \label{sec:Thm_Brauer_Manin}
We begin the proof of Theorem \ref{thm:Brauer_Manin} with some results
on Brauer-Severi schemes which correspond to \emph{algebraic} Brauer group elements.

Let $U$ be a smooth variety over a field $k$ of characteristic zero 
and let $V=V_1 \times_U \cdots \times_U V_r$ be a product
of Brauer-Severi schemes over $U$, with $[V_i] \in \Br_1 U$ for each $i$.
Fix a finite field extension $K$ of $k$ such that $[V_i] \otimes_k K =0 \in \Br U_{K}$ for each $i$.
We define $S$ to be the algebraic torus over $k$ given by the exact sequence
\begin{equation} \label{def:S}
	1 \to \Gm^r \to (\Res{K}{k}\Gm)^r \to S \to 1.
\end{equation}
Note that $S$ is rational since it may be embedded as an open subset of a product of projective spaces.
Applying \'{e}tale cohomology we obtain an exact sequence
\begin{equation} \label{eqn:S-torsor}
	(\Pic U_K)^r \to \HH^1(U,S) \stackrel{\delta}{\to} (\Br U)^r \to (\Br U_{K})^r.
\end{equation}
Here we have used the isomorphisms $\HH^j(U,\Res{K}{k} \Gm) \cong \HH^j(U_{K}, \Gm)$ for $j=1,2$,
from Shapiro's lemma.
As $([V_1],\ldots,[V_r])$ lies in the kernel of the last map in (\ref{eqn:S-torsor}),
there exists a (non-unique) $U$--torsor $W$ under $S$ whose
class $[W] \in \HH^1(U,S)$ maps to $([V_1],\ldots,[V_r])$. The specialisation maps
for the $V_i$ and $W$ are very closely related.

\begin{lemma}\label{lem:specialisation}
	Let $k \subset L$ be a field extension and let $u \in U(L)$. Then
	$$[W](u)=0 \in \HH^1(L, S_L) \text{ if and only if } ([V_1](u),\ldots, [V_r](u))=0 \in (\Br L)^r.$$
\end{lemma}
\begin{proof}
	The exact sequence \eqref{def:S} yields a commutative diagram with exact top row
	$$
	\xymatrix{0=\Pic (K \otimes_k L)^r \ar[r]&	\HH^1(L,S_L)  \ar[d] \ar[r]& (\Br L)^r \ar[d] \\
	& \HH^1(U_L,S_L)  \ar[r]& (\Br U_L)^r. }
	$$
	The top right arrow maps $W(u)$ to $([V_1](u),\ldots, [V_r](u))$, hence the result.
\end{proof}

Next, recall that a variety $V_1$ is said to be
\emph{stably birational} to a variety $V_2$ if there exists $n_1$ and $n_2$ such that
$V_1 \times \PP^{n_1}$ is birational to $V_2 \times \PP^{n_2}$.

\begin{lemma} \label{lem:stably_birational}
	$V$ is stably birational to $W$.
\end{lemma}
\begin{proof}
	Let $\eta$ denote the generic point of $U$.	As $V_\eta$ has a $k(V_\eta)$-point, 
	Lemma \ref{lem:specialisation} implies
	that $W_\eta$ has a $k(V_\eta)$-point. However $W_\eta$ is a $k(U)$-torsor under $S_{k(U)}$,
	hence, since $S_{k(U)}$ is rational, we see that $W_\eta \times_{k(U)} k(V_\eta)$
	is birational to $\PP^w \times_{k(U)} k(V_\eta)$ for some $w \in \NN$. 
	This implies that $W_\eta \times_{k(U)} V_\eta$ is birational
	to  $\PP^w\times_{k(U)} V_\eta$. As $V_\eta$ is a product of Brauer-Severi varieties, we may apply 
	a similar argument to deduce that $W_\eta \times_{k(U)} V_\eta$ is birational to $W_\eta \times_{k(U)} \PP^v$ for
	some $v \in \NN$, which completes the proof.
\end{proof}

We now specialise to the case where $U$ is a principal homogeneous
space under some algebraic torus over $k$ (keeping all previous notations).

\begin{lemma} \label{lem:W_is_PHS}
	$W$ is a principal homogeneous space under some algebraic torus.
\end{lemma}
\begin{proof}
	As $\overline{U}$ and $\overline{S}$  are algebraic tori and $W$ is a $U$-torsor under $S$,
	it follows from \cite[Thm.~5.6]{CT08} that $\overline{W}$ admits the structure of an
    algebraic group together with a short exact sequence
	\begin{equation}	\label{eqn:Is_solvable}
		1 \to \overline{S} \to \overline{W} \to \overline{U} \to 1.
	\end{equation}
	By (\ref{eqn:Is_solvable}), we see that $\overline{W}$ is
	solvable and contains no non-trivial unipotent elements.
	Therefore by the structure theorem for solvable
	algebraic groups \cite[Thm.~III.10.6]{Bor91}, we see that
	$\overline{W}$ is a torus.
	The fact that this implies that $W$ is itself a principal homogeneous space
	under some algebraic torus is well-known, see e.g.~\cite[Lem.~2.4.4]{Sko01}.
\end{proof}

Now let $k$ be a number field. It follows from \cite{San81}
and Lemma \ref{lem:W_is_PHS} that the Brauer-Manin obstruction is the only obstruction to the Hasse principle
and weak approximation for any smooth proper model of $W$.
We deduce the same for $V$, on applying Lemma \ref{lem:stably_birational} 
and a minor modification of \cite[Prop.~6.1]{CTPS13}
($\Brnr W / \Br_0 W$ is finite
as $W$ is geometrically rational \cite[Prop.~1.3.1]{CTSD94}).
This completes the proof of Theorem \ref{thm:Brauer_Manin}. \qed

\subsection{Subordinate Brauer group elements} \label{Sec:subordinate}
In the appendix of \cite[Ch.~II]{Ser97}, Serre defined the notion
of subordinate Brauer group elements for $\PP^1$. In this section,
we consider generalisations of this to other varieties.
Such Brauer group elements naturally arise in the calculation
of the leading constant in Theorem \ref{thm:Brauer}, and also assist
in the calculation of the unramified Brauer group of products of Brauer-Severi schemes.

\begin{definition}
	Let $X$ be a smooth proper variety over a field $k$ of characteristic zero and let $\br \subset \Br k(X)$
	be a finite subset. Then we say that $b \in \Br k(X)$ is \emph{subordinate} to $\br$
	with respect to $X$, if for each $D \in X^{(1)}$ the residue $\res_D(b)$ lies in 
	$\langle \res_D(\br) \rangle$. Let
	$$\Sub(X,\br)=\{ b \in \Br k(X) : \res_D(b) \in \langle \res_D(\br) \rangle \text{ for all } D \in X^{(1)}\},$$
	denote the group of all such elements.
\end{definition}
By (\ref{seq:purity0}) this is a subgroup of $\Br U$
for any open subset $U \subset X$ for which $\br \subset \Br U$.
Note that $\Br X = \Sub(X,0)$
is a subgroup of $\Sub(X,\br)$ of finite index, as $\br$ is finite.
It is important to note that $\Sub(X,\br)$ depends on the choice of model for $k(X)$ in general,
as the next example shows.

\begin{example} \label{exa:sub_bad}
	Consider $\mathbb{A}^2 \subset \PP^2$ and $\mathbb{A}^2 \subset \PP^1 \times \PP^1$
	embedded as a usual affine patch with coordinate functions $x$ and $y$.
	Suppose that $k$ contains a non-square $\alpha \in k^*$.
	Let $b \in \Br k(x,y)$ be the class of the quaternion algebra $(xy,\alpha)$.
	A simple residue calculation shows that 
    $(x,\alpha)$ is subordinate to $b$ with respect to $\PP^1 \times \PP^1$,
	but not with respect to $\PP^2$. In particular $\Sub(\PP^2,b) \neq \Sub(\PP^1 \times \PP^1,b)$.
 	In fact one can show that
	$$ \Sub(\PP^2,b) = \langle \Br k, (xy,\alpha) \rangle, \quad 
	\Sub(\PP^1 \times \PP^1,b) = \langle \Br k, (x,\alpha), (y,\alpha) \rangle.$$
\end{example}

We therefore also consider another definition which does not depend on the  model.
\begin{definition}
	Let $U$ be a smooth variety
	over $k$ and let $\br \subset \Br k(U)$ be a finite subset.
	We say that an element $b \in \Br k(U)$ is \emph{subordinate} to $\br$ with respect to $k(U)/k$,
	if for each discrete valuation $v$ of $k(U)$ which is trivial on $k$ the residue
	$\res_v(b)$ lies in $\langle \res_v(\br) \rangle$.	
	We let $\Sub(k(U)/k,\br)$ denote the group of all such elements.
\end{definition}
If $k$ is clear, we shall denote this by $\Sub(k(U),\br)$.
If $U=X$ is proper, then there is an inclusion $\Sub(k(X),\br) \subset \Sub(X,\br)$ and
$$\Brnr X = \Sub(k(X),0)=\Sub(X,0)=\Br X,$$
but however $\Sub(k(X),\br) \neq \Sub(X,\br)$ in general, as Example \ref{exa:sub_bad} shows.

As $\PP^1$ is the unique smooth projective variety
with function field $k(t)$, we see that for all finite subsets 
$\br \subset \Br k(t)$ we have $\Sub(k(t),\br)=\Sub(\PP^1,\br)$.
Colliot-Th\'{e}l\`{e}ne and Swinnerton-Dyer \cite[Thm.~2.2.1]{CTSD94}
considered groups of this type in the case of $\PP^1$.
If $V_1,\ldots,V_r$ are Brauer-Severi varieties over $k(t)$, then they showed that
(in our notation) there is a short exact sequence
$$ 0 \to \langle \br \rangle \to \Sub(k(t),\br) \to \Brnr(k(V)/k) \to 0,$$
where $\br=\{[V_1],\ldots,[V_r]\} \subset \Br k(t)$ and $V=V_1 \times_{k(t)} \cdots \times_{k(t)} V_r$.
Note that $\Brnr(k(V)/k)$ is isomorphic to the Brauer group of any smooth proper
model of $V$ over $k$. The following theorem is a generalisation of this result.

\begin{theorem} \label{thm:CTSD}
	Let $U$ be a smooth variety over a field $k$ of characteristic $0$
	and let $V_1,\ldots,V_r$ be a Brauer-Severi varieties over $k(U)$.
	Let $\br=\{[V_1],\ldots,[V_r]\} \subset \Br k(U)$ and 
	$V=V_1 \times_{k(U)} \cdots \times_{k(U)} V_r$. Then the composed morphism
	$$\Sub(k(U),\br) \subset \Br k(U) \to
	\Br V \subset \Br k(U)(V)=\Br k(V)$$
	has image $\Brnr(k(V)/k)$ and induces a short exact sequence
	\begin{equation} \label{eqn:CTSD}
		0 \to \langle \br \rangle \to \Sub(k(U),\br) \to \Brnr(k(V)/k) \to 0.
	\end{equation}
\end{theorem}
\begin{proof}
	Our proof is based on the proof of \cite[Thm.~2.2.1]{CTSD94}.
	Let $b_i = [V_i]$. A minor modification of a classical
	theorem of Amitsur (see e.g.~\cite[Thm.~5.4.1]{GS06})
	yields a short exact sequence
	\begin{equation} \label{eqn:Brauer_group_BS}
		0 \to \langle \br \rangle \to \Br k(U) \stackrel{\gamma}{\to} \Br V \to 0,
	\end{equation}
	where $\gamma$ denotes the usual restriction map.
	This shows that the sequence (\ref{eqn:CTSD}) is exact on the left and in the middle.
	Fix some $b \in \Sub(k(U),\br)$. We now show that $\gamma(b) \in \Brnr(k(V)/k)$.
	For all discrete valuation rings $k \subset R$ with field of fractions
	$k(U)$ and each $i=1,\ldots,r$, there exists some $m_{i,R} \in \ZZ$ such that
	\begin{equation} \label{eqn:res_R}
		\res_R(b) = \sum_{i=1}^{r}m_{i,R}\res_R(b_i).
	\end{equation}
	Let $k \subset R'$ be a discrete valuation ring with field of fractions
	$k(V)$. If $k(U) \subset R'$, then $\res_{R'}(\gamma(b))=0$.
	Otherwise $R=k(U) \cap R'$ is a discrete valuation ring containing $k$
	with field of fractions $k(U)$. Let $e(R'/R)$ denote the ramification degree of $R'$ over $R$.
	As $\gamma(b_i)=0$ for each $i=1,\ldots,r$ by (\ref{eqn:Brauer_group_BS}), we have
	$$\res_{R'}(\gamma(b)) = \res_{R'}(\gamma(b)) - \sum_{i=1}^{r}m_{i,R}\res_{R'}(\gamma(b_i)) =
	e(R'/R)\res_{R}\left(b - \sum_{i=1}^{r}m_{i,R}b_i\right)=0,$$
	on using \cite[Prop.~1.1.1]{CTSD94} followed by (\ref{eqn:res_R}).
	Hence $\gamma(b) \in \Brnr(k(V)/k)$.

	To finish the proof, it suffices to show that
	$\Sub(k(U),\br) \to \Brnr(k(V)/k)$ is surjective. Note that
	(\ref{eqn:Brauer_group_BS}) shows that every element of $\Brnr(k(V)/k)$
	has the form $\gamma(b)$ for some $b \in \Br k(U)$. So let $b \in \Br k(U)$ and suppose
	that $b \not \in \Sub(k(U),\br)$, i.e.~there exists a discrete valuation ring
    $k \subset R$ with field of fractions $k(U)$ such that
	$\res_R (b)$ does not lie in $\langle \res_R(\br) \rangle$.
	In order to show that $\gamma(b)$ does not lie in $\Brnr(k(V)/k)$,
	it suffices to construct a discrete valuation ring $k \subset R'$
	with field of fractions $k(V)$ such that $\res_{R'}(\gamma(b)) \neq 0$.
	To do this, we shall use a model $\mathcal{V} \to \Spec R$ for $V$
	from Lemma \ref{lem:AF}. Let $D$ denote one of the irreducible components
	of the special fibre of $\mathcal{V}$ and denote by $R'$ the corresponding discrete
	valuation ring. Let $v$ (resp.~$v'$) denote the discrete valuation on $R$ (resp.~$R'$).
	Applying the residue maps to the exact sequence (\ref{eqn:Brauer_group_BS}) we
	obtain a commutative diagram
	$$
	\xymatrix{	
	0 \ar[r]& \langle \br \rangle \ar[r] \ar[d]^{\res_R} & \Br k(U) \ar[r]^{\gamma}	 \ar[d]^{\res_R}	
    & \Br V \ar[r] \ar[d]^{\res_{R'}} & 0 \\
	0 \ar[r]& \langle \res_{R}(\br) \rangle \ar[r] & \HH^1(k(v),\QQ/\ZZ) \ar[r]&  \HH^1(k(v'),\QQ/\ZZ) &
	}
	$$
	with exact rows. Here the map $\HH^1(k(v),\QQ/\ZZ) \to  \HH^1(k(v'),\QQ/\ZZ)$ is the natural restriction
	map as $R \subset R'$ has ramification index one, since $D$ is integral (this follows from \cite[Prop.~1.1.1]{CTSD94}).
	The exactness of the bottom row	follows from a simple application of the inflation-restriction exact sequence,
	 on noting that by Lemma \ref{lem:AF},
	the algebraic closure of $k(v)$ in $k(v')$ is exactly the compositum of the field extensions given by the $\res_R(b_i)$.
	By assumption $\res_R (b)$ does not lie in $\langle \res_R(\br) \rangle$,
	hence we see that $\res_{R'}(\gamma(b)) \neq 0$. This completes the proof.
\end{proof}

Next recall the following theorem of Harari \cite[Thm.~2.1.1]{Har94}.

\begin{theorem}[Harari] \label{thm:Harari}
    Let $U$ be a smooth quasi-projective  variety over a number field $F$ and let $b \in \Br U$.
    Then $b \in \Brnr U$ if and only if the map $U(F_v) \to \Br F_v$
    induced by $b$ is trivial for all but finitely many places $v$ of $F$.
\end{theorem}

It is the ``if'' part of this theorem which is non-trivial; 
the ``only if'' part is classical, as explained in Section \ref{sec:Br_fields}.
We now give a generalisation of Harari's theorem
to subordinate Brauer group elements, which makes clear their relevance to the study of
zero-loci of Brauer group elements.

\begin{theorem}	\label{thm:Sub_Harari}
	Let $U$ be a smooth quasi-projective variety over a number field $F$, let
	$\br \subset \Br U$ be a finite subset
	and let $b \in \Br U$.
    Then $b \in \Sub(F(U),\br)$ if and only if
	the map $U(F_v)_{\br} \to \Br F_v$ induced by $b$ is
	trivial for all but finitely many places $v$ of $F$.
\end{theorem}
\begin{proof}
	Let $\pi:V \to U$ be the product of Brauer-Severi schemes corresponding to the elements of $\br$.
	For any place $v$ of $F$ we have a commutative diagram
	$$
	\xymatrix{V(F_v)  \ar[d] \ar[r]& \Br F_v \\
	U(F_v)_\br \ar[ur]& }
	$$
	given by pairing with $\pi^*b$ and $b$, respectively. Note that,
	by definition, the map $V(F_v) \to U(F_v)_\br$ is surjective.
	By Theorem \ref{thm:Harari} we have
	$\pi^*b \in \Brnr(F(V)/F)$ if and only if the evaluation map $V(F_v) \to \Br F_v$
    induced by $\pi^*b$ is trivial for all but finitely many places $v$ of $F$.
	However by Theorem \ref{thm:CTSD}, we also know that $\pi^*b \in \Brnr(F(V)/F)$
	if and only if $b \in \Sub(F(U),\br)$. This proves the required equivalence.
\end{proof}

We next show that under suitable conditions,
Brauer group elements which are subordinate to \emph{algebraic}
Brauer group elements are themselves algebraic.
\begin{lemma}\label{lem:sub_algebraic}
	Let $X$ be a smooth proper variety over a field $k$ of characteristic zero and 
	let $U \subset X$ be a dense open subset.
	Let $\br \subset \Br_1 U$ be a finite subset and suppose that $\Br \overline{X} = 0$. Then
	$\Sub(X,\br) \subset \Br_1 U$.
\end{lemma}
\begin{proof}
	Let $b \in \Sub(X,\br)$. As $\br \subset \Br_1 U$, each element of $\br$ has trivial residues over $\kbar$,
	so the same holds for $\overline{b} = b \otimes_k \kbar$ (this follows from \cite[Prop.~1.1.1]{CTSD94}, for example). Hence by purity $\overline{b} \in \Br \overline{X}=0$,
	as required.
\end{proof}

We finish by giving some necessary conditions for the finiteness of $\Sub(X,\br)/\Br_0 X$.

\begin{lemma}\label{lem:Sub_is_finite}
	Let $X$ be a smooth proper geometrically rational variety over a field $k$ of characteristic zero.
	Let $\br \subset \Br k(X)$ be finite. Then $\Sub(X,\br)/\Br_0 X$ is finite.
\end{lemma}
\begin{proof}
	As $\Br X \subset \Sub(X,\br)$ is a subgroup of finite index, it suffices to show that the group $\Br X/\Br_0 X$
	is finite. This follows from the assumption that $X$ is geometrically rational (see e.g.~\cite[Prop.~1.3.1]{CTSD94}).
\end{proof}
Note that if $U$ and $V$ are as in Theorem \ref{thm:CTSD} and $U$ is geometrically rational,
then Lemma \ref{lem:Sub_is_finite} implies that $\Brnr(k(V)/k)/\Br_0 V$ is finite.

\section{Hecke \texorpdfstring{$L$}{L}-functions and Delange's Tauberian theorem} \label{sec:L-functions}
In this section we gather various analytic results on virtual
$L$-functions and certain partial Euler products which shall
arise in the proof of Theorem \ref{thm:Brauer}, together with a Tauberian
theorem due to Delange \cite{Del54}.

\subsection{Hecke $L$-functions} \label{sec:Hecke}
Let $F$ be a number field. Recall (see e.g.~\cite{CF10} or \cite{Wei74})
that a Hecke character for $F$ is a character $\Adele_F^* \to S^1 \subset \CC^*$
which is trivial on $F^* \subset \Adele_F^*$. Each Hecke character $\chi$ may
be decomposed as a product of local characters $\chi_v:F_v^* \to S^1$ for each place $v$ of $F$.
We say that $\chi$ is \emph{unramified} at $v$ if the character $\chi_v$ is trivial on $\OO_v^*$.
Each Hecke character has a conductor $q(\chi) \in \NN$, which measures the ramification of $\chi$
at the finite places. The $L$-function of $\chi$ is defined to be
$$L(\chi,s) = \prod_v \left(1-\frac{\chi_v(\pi_v)}{q_v^s}\right)^{-1},$$
where the Euler product is taken over those non-archimedean places $v$ of $F$ for which $\chi_v$ is unramified.
We denote by $\zeta_F(s)=L(1,s)$ the Dedekind zeta function of $F$.

These $L$-functions admit a meromorphic continuation to $\CC$, and are holomorphic
if $\chi$ is non-principal (we say that a Hecke character
is principal if its restriction to the collection $ \Adele_F^{*1}$  of norm one ideles is trivial).
The principal characters are exactly those of the form $|| \cdot ||^{i\theta}$
for some $\theta \in \RR$, where
$$||\cdot||: \Adele_F^* \to S^1, \qquad (t_v) \mapsto \prod_{v \in \Val(F)}|t_v|_v,$$
is the adelic norm map, and $L(|| \cdot ||^{i\theta},s)$ admits a single pole
of order $1$ at $s=1+i\theta$.

We now state bounds for the growth rate of Hecke $L$-functions. Let $\chi$ be a character of $\prod_{v \mid \infty} F_v^*$
(e.g.~$\chi$ could come from a Hecke character). Let $v$ be an archimedean place of $F$.
Restricting $\chi$ to the obvious subgroup $\RR_{>0} \subset F_v^*$ (as $F_v^* \cong \RR^*$ or $\CC^*$),
we obtain a continuous homomorphism $\RR_{>0} \to S^1$. Such a homomorphism must be of the form $x \mapsto |x|^{i\kappa_v}$
for some $\kappa_v \in \RR$. We then define
\begin{equation} \label{def:character_norm}
	||\chi|| = \max_{v \mid \infty} |\kappa_v|.
\end{equation}

\begin{lemma}\label{lem:Hecke_bound}
	Let $\varepsilon >0$, let $C$ be a compact subset of the half plane $\re s \geq 1$,
	and let $\chi$ be a non-principal Hecke character of $F$. Then 
	$$L(\chi,s) \ll_{\varepsilon,C} q(\chi)^\varepsilon(1+ ||\chi||)^{\varepsilon},
	\quad (s-1)\zeta_F(s) \ll_{\varepsilon,C} 1, \quad s \in C.$$
\end{lemma}
\begin{proof}
	The first bound follows from \cite[(5.20) p.~100]{IK04}. The second is trivial.
\end{proof}
Note that the non-trivial principal Hecke characters were overlooked in the statement of \cite[Thm.~3.2.3]{BT95}.

\subsection{Analytic properties of certain partial Euler products}
\label{sec:partial_zeta_functions}
Let $F$ be a number field and fix a finite group $\R$ of Hecke characters for $F$. For each place $v$ of $F$, let
\begin{align*}
	\thorn_{v}: F_v^* &\to \{0,1\}, \quad
	\thorn_{v}: t_v \mapsto \left\{
	\begin{array}{ll}
	1, \quad \text{if } \rho_v(t_v)=1 \text{ for all } \rho \in \R,\\
	0, \quad \text{otherwise},
	\end{array}\right.
\end{align*}
be the indicator function of $\bigcap_{\rho \in \R} \ker \rho_v$.
By character orthogonality we have
\begin{equation} \label{eqn:character_orthog}
	\thorn_{v} = \frac{1}{|\R|} \sum_{\rho \in \R} \rho_v.
\end{equation}
Let $\chi$ be a Hecke character of $F$. The partial Euler product
of interest to us is
\begin{equation}\label{def:partial_zeta_function}
	L_{\R}(\chi,s)=\prod_{v} \left(1-\frac{\thorn_v(\pi_{v})\chi_{v}(\pi_{v})}{q_v^{s}}\right)^{-1}
	=\prod_{\substack{v \\ \rho_v(\pi_v) = 1  \\ \forall \rho_v \in \R}} \left(1-\frac{\chi_{v}(\pi_{v})}{q_v^{s}}\right)^{-1},
\end{equation}
where the products are taken over those non-archimedean places $v \in \Val(F)$ for
which $\chi_v$ and $\rho_v$ are unramified for all $\rho \in \R$.
Note that if $\R = \{1\}$ then we have $L_{\{1\}}(\chi,s) = L(\chi,s)$. 
For general $\R$, the right-hand Euler product in (\ref{def:partial_zeta_function})
is over a certain collection of places of density $1/|\R|$.
It is clear that the products in (\ref{def:partial_zeta_function}) are
absolutely convergent for $\re s >1$ and define a holomorphic function
without zeros on this domain. The analytic properties of functions of this 
type have been studied by numerous authors
(see e.g.~\cite{Has10}). We shall content
ourselves with the following.

\begin{lemma} \label{lem:partial_Euler_product}
    Let $\varepsilon >0$ and let $\chi$ be a Hecke character for $F$. Then there exists a function $ G(\R,\chi,s)$
	which is holomorphic, uniformly bounded with respect to $\chi$ and non-zero
	on the domain $\re s \geq 1/2 + \varepsilon$ such that
	$$L_{\R}(\chi,s)^{|\R|} =G(\R,\chi,s) \prod_{\rho \in \R} L(\rho\chi,s),$$
	for $\re s >1$.
	
	If $\chi$ is not of the form $||\cdot||^{i\theta}\rho$ for some $\theta \in \RR$ and some $\rho \in \R$, then
	$L_{\R}(\chi,s)$ admits a holomorphic continuation to the line $\re s =1$.
	If $\rho \in \R$, then $L_{\R}(\rho,s)$ admits a holomorphic continuation to the line $\re s =1$, away from $s=1$.
	Here we have
	$$L_{\R}(\rho,s) = \frac{c_{\R,\rho}}{(s-1)^{1/|\R|}} + O\left(\frac{1}{(s-1)^{1/|\R|-1}}\right),$$
	as $s \to 1$, where  $c_{\R,\rho} \neq 0$.
\end{lemma}
\begin{proof}
	To prove the first equality, it suffices to compare the  Euler products. For $\re s \geq 1/2 + \varepsilon$
	and unramified $v \in \Val(F)$, the Euler factor at $v$ in the product of $L$-functions on the right-hand side is
	\begin{align*}
		\prod_{\rho \in \R}\left(1 - \frac{\rho_v(\pi_v)\chi_v(\pi_v)}{q_v^{s}}\right)^{-1}&=
		 1 + \frac{\chi_v(\pi_v)}{q_v^s} \sum_{\rho \in \R}\rho_v(\pi_v) + O(q_v^{-1-\varepsilon}).
	\end{align*}
	The first part therefore follows by (\ref{eqn:character_orthog}). The second part
	of the lemma follows from the analytic properties of the functions 
	$L(\rho \chi,s)^{1/|\R|}$, which one obtains from the analytic properties 
	of Hecke $L$-functions stated in \S \ref{sec:Hecke}, together with the fact that Hecke $L$-functions
	do not vanish for $\re s \geq 1$ \cite[Thm.~5.10]{IK04}  (the reader who is unfamiliar with rational powers
	of $L$-functions is advised to consult \cite[Ch.~II.5.1]{Ten95}).
\end{proof}

\subsection{Delange's Tauberian theorem}
We now state the version of Delange's Tauberian theorem which shall 
be used in the proof of Theorem \ref{thm:Brauer}. 

\begin{theorem} \label{thm:Delange}
    Let $f(s) = \sum_{n=1}^\infty a_n/n^s$ be a Dirichlet series with real non-negative coefficients
    which converges for $\re s > 1$. Suppose that there exists some real number
	$\omega >0$ and some $\delta >0$ such that the function $g(s)=f(s)(s-1)^{\omega}$
    admits an extension to an infinitely differentiable function on the line $\re s =1$ with $g(1) \neq 0$ and that
	$$f(s) = \frac{g(1)}{(s-1)^\omega} + O\left(\frac{1}{(s-1)^{\omega - \delta}}\right), \quad \text{as } s \to 1.$$
	Then
    $$ \sum_{n \leq x} a_n \sim \frac{g(1)}{\Gamma(\omega)} x(\log x)^{\omega-1},\quad \text{as } x \to \infty.$$
\end{theorem}
\begin{proof}
 	This is almost the statement of \cite[Thm.~III]{Del54}, however there Delange assumes moreover
	that $f$ be holomorphic on the line $\re s = 1$. However the proof follows
	with minor modifications to the proof of \cite[Thm.~III]{Del54},
	as a standard application of the more general Tauberian theorem \cite[Thm.~I]{Del54}.
	The details are left to the reader (Delange takes $\Psi(u)=1$ in \emph{ibid.}
	and the exact same proof works in our case on taking $\Psi(u)=u^{\delta-1}$.)
\end{proof}

\section{Algebraic tori, toric varieties and their Brauer groups} \label{sec:Tori}
In this section we gather various facts about algebraic tori
and toric varieties over number fields. The main results here are
a description of algebraic Brauer groups of algebraic tori over number
fields and an analogue for subordinate
Brauer group elements of a theorem of Voskresenski\v{\i} \cite{Vos70}.
We finish by studying heights on toric varieties.

\subsection{Algebraic tori over perfect fields}
Let $k$ be a perfect field.
Recall that an algebraic torus over $k$ is an algebraic group
$T$ over $k$ such that $\Tbar=T \times_k \kbar$ is isomorphic to $\Gm^n$, for some $n \in \NN$.
We denote by $1 \in T(k)$ the identity element of $T$.

The category of algebraic
tori is dual to the category of free $\mathbb{Z}$-modules with continuous $G_k$-action.
This is given by associating to an algebraic torus $T$ its character group
$X^*(\Tbar)=\Hom(\Tbar,\Gm)$. Note that $X^*(T)=X^*(\Tbar)^{G_k}$ is the
collection of characters of $\Tbar$ which are defined over $k$.
We denote by $X_*(T) = \Hom(X^*(T),\ZZ)$ the collection of cocharacters of $T$
and also let $X^*(T)_\RR=X^*(T) \otimes_\ZZ \RR$ and $X_*(T)_\RR=X_*(T) \otimes_\ZZ \RR$.
The splitting field of $T$ is the fixed field of the
kernel of the representation $G_k \to \GL(X^*(\Tbar))$; it is the smallest Galois  extension of
$k$ over which $T$ becomes isomorphic to $\Gm^n$.

\subsection{Algebraic tori over number fields}
The standard references for this section are the papers \cite{Ono61} and \cite{Ono63}.
Many of the facts presented here are natural generalisations of the
case of $\Gm$ studied in Tate's thesis \cite[Ch.~XV]{CF10}.
\subsubsection{The local points} \label{sec:tori_local_space}
Let $T$ be an algebraic torus over a number field $F$. For any place $v$ of $F$ we shall denote by
$T(\OO_v)$ the maximal compact subgroup of $T(F_v)$.
For non-archimedean $v$, we have a bilinear pairing
$$T(F_v) \times X^*(T_v) \to \ZZ, \quad (t,m) \mapsto \frac{\log|m(t)|_v}{\log q_v}.$$
This pairing induces an exact sequence
\begin{equation} \label{seq:local_finite}
	0 \to T(\OO_v) \to T(F_v) \to X_*(T_v).
\end{equation}
The image of $T(F_v)\to X_*(T_v)$ is open of finite index, and 
this map is surjective if $v$ is unramified in the splitting field of $T$.
For archimedean $v$ we have a similar pairing
$$T(F_v) \times X^*(T_v)_{\RR} \to \RR, \quad (t,m) \mapsto \log|m(t)|_v,$$
which induces a short exact sequence
\begin{equation} \label{seq:local_infinite}
	0 \to T(\OO_v) \to T(F_v) \to X_*(T_v)_\RR \to 0.
\end{equation}
For archimedean $v$, the maps $T(F_v) \to X_*(T_v)_\RR$ and $T(F_v) \to X_*(T)_\RR$ admit canonical sections.
The section to the first map is constructed in \cite[Lem.~2.18]{Bou11}. For the second map,
it suffices to construct a canonical section of the map $X_*(T_v)_\RR \to X_*(T)_\RR$.
This may be given by the dual of the map
\begin{align*}
	X^*(T_v)_\RR &\to X^*(T)_\RR, \quad
	m \mapsto \frac{1}{|\Gamma|}\sum_{\gamma \in \Gamma} m^\gamma,
\end{align*}
where $\Gamma$ is the Galois group of the splitting field of $T$. We also let
$T_\infty = \prod_{v|\infty}T_v$.

\subsubsection{The adelic space} \label{sec:tori_adelic_space}
The local pairings give rise to an adelic pairing
$$T(\Adele_F) \times X^*(T)_\RR \to \RR,
\quad ((t_v),m) \mapsto \sum_{v \in \Val(F)} \frac{\log|m(t)|_v}{\log q_v},$$
where we take $\log q_v = 1$ if $v$ is archimedean.
If we denote by $T(\Adele_F)^1$ the left kernel of this pairing,
we have a short exact sequence
\begin{equation} \label{seq:adelic_splitting}
	0 \to T(\Adele_F)^1 \to T(\Adele_F) \to X_*(T)_\RR \to 0.
\end{equation}
Under the diagonal embedding $T(F)$ is a discrete and cocompact subgroup of $T(\Adele_F)^1$.
The sequence (\ref{seq:adelic_splitting}) admits a splitting which is natural in the sense of category theory,
though a choice still needs to be made. As explained in the previous section, for $v$ archimedean we may
canonically identify $X_*(T)_\RR$ with a subgroup of $T(F_v)$. We therefore choose the section given by
\begin{align*}
	X_*(T)_\RR &\to T(\Adele_F) , \quad
	n \mapsto (\epsilon_v n/[F:\QQ])_{v|\infty} \times (1)_{v \nmid \infty},
\end{align*}
where $\epsilon_v = 1$ if $v$ is real and $\epsilon_v =2$ is $v$ is complex. This choice gives a functorial isomorphism
\begin{equation} \label{eqn:adelic_splitting}
    T(\Adele_F) \cong T(\Adele_F)^1 \times X_*(T)_\RR.
\end{equation}

\subsubsection{The Hasse principle and weak approximation}
Let
$$\Sha(T) = \ker \left(\HH^1(F,T) \to \prod_{v \in \Val(F)} \HH^1(F_v,T)\right),$$
denote the Tate-Shafarevich group of $T$. 
This is finite and Sansuc \cite[Prop.~8.3]{San81} constructed a canonical isomorphism
\begin{equation} \label{eqn:Be_Sha}
    \Sha(T) \cong \Be(T)^\sim,
\end{equation}
where $\Be(T)$ is given by (\ref{def:Be}). Next let
\begin{equation} \label{def:A(T)}
	A(T)= T(\Adele_F)/ \overline{T(F)}^w,
\end{equation}
where $\overline{T(F)}^w$ denotes the closure of $T(F)$ in $T(\Adele_F)$
with respect to the product topology.
This is finite and by a theorem of Voskresenski\v{\i} \cite{Vos70} (see also \cite[Thm.~9.2]{San81})
there is a short exact sequence
\begin{equation} \label{seq:Br_Sha}
    0 \to \Be(T) \to \Brnr T / \Br F \to A(T)^\sim \to 0,
\end{equation}
where $\Be(T)$ is given by \eqref{def:Be}.
Here we have used the isomorphisms
$\Brnr T / \Br F \cong \HH^1(F,\Pic \overline{X}) \cong \Br_a X$,
which hold for any smooth proper model $X$ of $T$ (see e.g.~\cite[Lem.~6.3(iii)]{San81}).

\subsubsection{Characters}
Given a place $v$ of $F$ and a character $\chi_v$ of $T(F_v)$, we shall say that
$\chi_v$ is \emph{unramified} if it is trivial on $T(\OO_v)$.
We say that a character $\chi$ of $T(\Adele_F)$ is \emph{automorphic} if it is trivial on $T(F)$.
Note that the automorphic characters of $\Gm(\Adele_F)=\Adele_F^*$ are exactly the Hecke characters of $F$.
The canonical sections of the composition $T(\Adele_F) \to \prod_{v \mid \infty} T(F_v) \to X_*(T_\infty)_\RR$
explained in Section \ref{sec:tori_local_space} give rise to a ``type at infinity'' map
\begin{equation}\label{def:type_at_infinity}
	\Fourier{T(\Adele_F)} \to X^*(T_\infty)_\RR.
\end{equation}
If we let $K_T=\prod_{v \in \Val(F)}T(\OO_v) \subset T(\Adele_F)^1$, then the splitting (\ref{eqn:adelic_splitting}) induces a map
$$\Fourier{\left(T(\Adele_F)^1/T(F)K_T\right)} \to X^*(T_\infty)_\RR.$$
This has finite kernel and moreover its image is a lattice of codimension $\rank X^*(T)$ \cite[Lem.~4.52]{Bou11}
(this may be viewed as a generalisation of finiteness of the class number and Dirichlet's unit theorem).

\subsubsection{The Haar measure and the Tamagawa number} \label{sec:Haar_measure_Tamagawa}
Let $\omega$ be an invariant differential form on $T$. By a classical construction (see e.g.~\cite[\S2.1.7]{CLT10}),
for each place $v$ of $F$ we obtain a measure (denoted $|\omega|_v$) on $T(F_v)$,
which is a Haar measure on $T(F_v)$. Let
$$\mu_v = c_v^{-1}|\omega|_v, \qquad c_v = \left \{
	\begin{array}{ll}
		L_v(X^*(\Tbar),1)^{-1},& \quad v \text{ non-archimedean}, \\
		1,& \quad v \text{ archimedean}, \\
	\end{array}\right.$$	
where $L(X^*(\Tbar),s)$ is the associated Artin $L$-function.
The product of the $\mu_v$ converges to yield
a Haar measure $\mu$ on $T(\Adele_F)$ (see \cite[\S3.3]{Ono61}).  
By the product formula, this measure is independent of the choice of
$\omega$. By the splitting (\ref{eqn:adelic_splitting})
we obtain a Haar measure $\mu^1$ on $T(\Adele_F)^1$, on equipping $X_*(T)_\RR$ with the
unique Haar measure such that $X_*(T) \subset X_*(T)_\RR$ has covolume one.
Ono's works \cite{Ono61, Ono63}
imply that 
\begin{equation} \label{eqn:Ono}
	\vol(T(\Adele_F)^1/T(F)) = L^*(X^*(\Tbar),1) \cdot \frac{|\Pic T|}{|\Sha(T)|},
\end{equation}
where $$L^*(X^*(\Tbar),1) = \lim_{s \to 1}(s-1)^{\rank X^*(T)}L(X^*(\Tbar),s).$$

\subsection{Algebraic Brauer groups of algebraic tori}
\subsubsection{Brauer groups over perfect fields}
Let $T$ be an algebraic torus over a perfect field $k$. In this paper, we shall be particularly
interested in the group
\begin{equation}
    \Br_e T = \{ b \in \Br_1 T: b(1)=0\}. \label{def:Br_e}
\end{equation}
There are canonical isomorphisms $\Br_1 T \cong \Br_0 T \bigoplus \Br_e T$
and $\Br_e T \cong \Br_a T$.

\begin{lemma}\label{lem:Pic_Br_iso}
	There are natural isomorphisms
	$$\Pic T \cong \HH^1(k,X^*(\Tbar)), \quad \Br_e T \cong \HH^2(k,X^*(\Tbar)).$$
	In particular $\Pic T$ is finite.
\end{lemma}
\begin{proof}
	This is a standard application of the Hochshild-Serre spectral sequence, see e.g.
	\cite[Lem.~6.9(ii)]{San81}.
\end{proof}

Our interest in $\Br_e T$ stems from the following result of Sansuc, which is
one of the key reasons why we focus on algebraic Brauer group elements only in Theorem \ref{thm:Brauer}.

\begin{lemma}\label{lem:Br_e_homo}
	The pairing
	\begin{align*}
		\Br_e T \times T(k) \to \Br k, \qquad
		(b,t) \mapsto b(t),
	\end{align*}
	is bilinear.
\end{lemma}
\begin{proof}
	This is a special case of \cite[Lem.~6.9]{San81}.
\end{proof}

There is another description of this pairing which we shall need.
Namely, as noted by Sansuc at the top of page 65 of \cite{San81}, we have a commutative diagram

\begin{equation}\label{diag:Br_commute}
\begin{split}
\xymatrix@C=0.1pt{
	\Br_e T\ar[d]  &\times 						& T(k) \ar[d] \ar[rrrrr] &  & & & & \Br k \ar@{=}[d] \\
	\HH^2(k, X^*(\Tbar))  & \times &  \HH^0(k,T(\kbar))  \ar[rrrrr]^{\smile} & & & & & \HH^2(k,\kbar^*)  }
\end{split}
\end{equation}
where $\smile$ denotes the cup product and the map $\Br_e T \to \HH^2(k, X^*(\Tbar))$ is the natural isomorphism
from Lemma \ref{lem:Pic_Br_iso}.

\subsubsection{Brauer groups over number fields}
We now give a description of the algebraic Brauer groups of tori over number fields
and their completions, via the Brauer pairing.
We begin with an elementary lemma on $\QQ/\ZZ$-duality and completions. 

\begin{lemma} \label{lem:completion}
	Let $G$ be a topological group. Assume that every element of
	$G^\sim$ has finite order. Then the natural map
	\begin{equation*} 
		(\widehat{G})^\sim \to G^\sim,
	\end{equation*}
	is an isomorphism.
\end{lemma}
\begin{proof}
	This follows from the observation that the map $G \to \widehat{G}$
	has dense image and 
	induces a quotient-preserving bijection between the finite-index normal open subgroups of $G$ and $\widehat{G}$.
\end{proof}

The assumptions of Lemma \ref{lem:completion} hold if $G$ is Hausdorff and an extension of a compact group by a finite
rank $\ZZ$- or $\RR$-module. In particular, let $T$ be an algebraic torus over a number field $F$.
Then the sequences \eqref{seq:local_finite}, \eqref{seq:local_infinite} and \eqref{seq:adelic_splitting}
imply that $T(\Adele_F)/T(F)$ and $T(F_v)$ satisfy the assumptions of Lemma \ref{lem:completion}
for all $v \in \Val(F)$. 


\begin{theorem}\label{thm:Brauer_local}
	For any place $v$ of $F$, the bilinear pairing
	$$\Br_e T_v \times T(F_v)  \to \Br F_v \subset \QQ / \ZZ,$$
	induces an isomorphism
	$$\Br_e T_v \cong T(F_v)^\sim,$$
	of abelian groups.
\end{theorem}
\begin{proof}
	It follows from local Tate duality that the pairing
	$$\HH^2(F_v,X^*(\Tbar)) \times \widehat{T(F_v)} \to \QQ/\ZZ,$$
	is perfect (see \cite[Cor.~I.2.4]{Mil06}
	for the non-archimedean case and \cite[Thm.~I.2.13]{Mil06} 
	for the archimedean case). In the light of (\ref{diag:Br_commute}), 
	the result follows from Lemma \ref{lem:completion}.
\end{proof}

\begin{theorem}\label{thm:Brauer_global}
	The pairing
	\begin{align*}
	    \Br_e T \times T(\Adele_F) \to \QQ / \ZZ,
	\end{align*}
	is bilinear and induces a short exact sequence
	$$ 0 \to \Be(T) \to \Br_e T \to (T(\Adele_F)/T(F))^\sim \to 0,$$
	where $\Be(T)$ is given by (\ref{def:Be}).
\end{theorem}
\begin{proof}
	Bilinearity follows from applying Lemma \ref{lem:Br_e_homo} to the local pairings.
	To derive the exact sequence, we shall use Nakayama duality \cite[Cor.~I.4.7]{Mil06}. 
	Let 
	$$C_{\Fbar} = \varinjlim_{F \subset E} \Adele_E^*/E^*, \qquad T(C_{\Fbar})=\Hom(X^*(\Tbar), C_{\Fbar}),
	\qquad G(T)=\HH^0(F,T(C_{\Fbar})).$$
	Then Nakayama duality implies that the pairing
	$$\HH^2(F,X^*(\Tbar)) \times \widehat{G(T)} \to \QQ/\ZZ,$$
	given by composing the cup product with the natural surjection $\HH^2(F,C_{\Fbar}) \to \QQ/\ZZ$,
	is perfect.	Next, as the natural inclusion $\Adele_F^* /  F^* \subset C_{\Fbar}$ has Galois
	invariant image, we deduce from (\ref{diag:Br_commute}) a commutative diagram
	\begin{equation} \label{eqn:Br_commute}
	\begin{split}
	\xymatrix@C=0.1pt{	\Br_e T  \ar[d]&\times & T(\Adele_F)/T(F) \ar[d] \ar[rrrrr] &  & & & & \QQ/\ZZ \ar@{=}[d]\\
	\HH^2(F, X^*(\Tbar))  & \times &  G(T)  \ar[rrrrr] & & & & & \QQ/\ZZ .} 
	\end{split}
	\end{equation}
	Let us now use Lemma \ref{lem:completion} to show that 
	\begin{equation} \label{eqn:G(T)}
		(\widehat{G(T)})^\sim \cong G(T)^\sim.
	\end{equation}
	To see this, note that $T(\Adele_F)/T(F)$ satisfies the assumptions of 
	Lemma \ref{lem:completion} and that
	$T(\Adele_F)/T(F)$ is a closed
	subgroup of $G(T)$ of finite index (see \cite[eq.~(2.3.5)]{Bou11}). From this
	it follows that $G(T)$ also satisfies the assumptions of Lemma \ref{lem:completion},
	as required.
	
	Combining \eqref{eqn:Br_commute} with \eqref{eqn:G(T)}, Nakayama duality
	and the fact that $T(\Adele_F)/T(F) \subset G(T)$ is closed of finite index,
	we deduce that the map $\epsilon: \Br_e T \to (T(\Adele_F)/T(F))^\sim$ is surjective.
	Next let $b \in \ker \epsilon$, so that its induced character of $T(\Adele_F)$ is trivial.
	Then Theorem \ref{thm:Brauer_local}  implies that $b \otimes F_v =0$ for each $v \in \Val(F)$,
	i.e.~$b \in \Be(T)$. As clearly $\Be(T) \subset \ker \epsilon$, the result is proved.
\end{proof}

\begin{corollary}\label{cor:rational}
	Suppose that $T$ is rational. Then $\Be(T)=0$ and the Brauer pairing
	induces an isomorphism 
	$$\Br_e T \cong \left(T(\Adele_F)/T(F)\right)^\sim.$$
\end{corollary}
\begin{proof}
	As $\Be(T)$ is a birational invariant \cite[Lem.~6.1]{San81},
	we have $\Be(T) = \Be(\PP^n_F)=\Be(F)=0$.
    The result then follows from Theorem \ref{thm:Brauer_global}.
\end{proof}

\subsection{Toric varieties}
\label{sec:toric_varieties}
Let $T$ be an algebraic torus over a perfect field $k$. In this paper a toric variety for $T$ is a
smooth projective variety $X$ with a faithful action of $T$ that has an open dense orbit
which contains a rational point. The complement of this orbit is a divisor, whose irreducible
components we call the \emph{boundary components} of $X$.
We follow the approach to toric varieties taken
in \cite{CLT14}.

\subsubsection{The boundary components}
Fix a toric variety $X$ for an algebraic torus $T$ over a number field $F$.
Let $\overline{\A}$ denote the set of boundary components of $\overline{X}$ and similarly
define $\A$ (resp.~$\A_v$) to be the set of boundary components of $X$ (resp.~$X_v$ for a place $v$ of $F$).
Given $\alpha \in \A$, we let $D_{\alpha}$ denote the corresponding irreducible divisor in $X$, we define
$F_{\alpha}= \Fbar \cap F(D_{\alpha}) \subset \Fbar(D_\alpha)$
and let $f_\alpha=[F_\alpha:F]$.
For $\overline{\alpha} \in \overline{\A}$ and $\alpha_v \in \A_v$, we define $D_{\overline{\alpha}},
D_{\alpha_v}, F_{\alpha_v}$ and $f_{\alpha_v}$ similarly.
Given a place $v$ of $F$, we say that an element
$\alpha_v \in \A_v$ \emph{divides} an element $\alpha \in \A$ (written $\alpha_v \mid \alpha$) if
$D_{\alpha_v} \subset D_{\alpha}$. Note that for $v$ non-archimedean and $\alpha \in \A$, there is a bijective
correspondence between those $\alpha_v \in \A_v$ such that $\alpha_v \mid \alpha$ and those places
$w$ of $F_\alpha$ such that $w \mid v$. There is an isomorphism
\begin{equation} \label{eqn:canonical_bundle}
	\omega_X \cong \bigotimes_{\alpha \in \A} \OO_X(-D_\alpha),
\end{equation}
where $\omega_X$ denotes the canonical bundle of $X$ (over $\CC$ this is \cite[Thm.~8.2.3]{CLS11},
and a similar proof works in our setting).

\subsubsection{The Picard group and the Brauer group}
Associated to each toric variety $X$, we have the following fundamental short exact sequence of Galois modules (see e.g.~\cite[Thm.~4.2.1]{CLS11}),
\begin{equation} \label{eqn:fundamental}
	0 \to X^*(\Tbar) \to \ZZ^{\overline{\A}} \to \Pic \overline{X} \to 0,
\end{equation}
where  $\ZZ^{\overline{\A}}$ denotes the free abelian group generated by the elements of $\overline{\A}$.
By duality for tori, we deduce from  (\ref{eqn:fundamental}) the following short exact sequence
\begin{equation}\label{eqn:fundamental_dual}
	0 \to T_{\NS} \to \prod_{\alpha \in \A} T_{\alpha} \to T \to 0.
\end{equation}
Here $T_{\NS}$ denotes the N\'{e}ron-Severi torus of $X$ and $T_{\alpha}=\Res{F_{\alpha}}{F} \Gm$ is the Weil restriction of $\Gm$
with respect to $F \subset F_\alpha$.
Given an element $b \in \Br T$, we may pull it back to obtain an element $b_{\alpha} \in \Br T_{\alpha}$.
Also, for a character $\chi \in \Fourier{T(\Adele_F)}$ we denote by $\chi_\alpha$ its image in $\Fourier{T_\alpha(\Adele_F)}$.
If $\chi$ is automorphic, then we will often identify $\chi_\alpha$ with a Hecke character of $F_\alpha$ via the canonical
isomorphism $T_\alpha(\Adele_F) = \Adele_{ F_\alpha}^*$.
Applying Galois cohomology to (\ref{eqn:fundamental}) we obtain a long exact sequence, the first part of which reads
\begin{equation} \label{seq:toric_Pic}
	0 \to X^*(T) \to \ZZ^{\A} \to \Pic X \to \Pic T \to 0.
\end{equation}
Here we have used the isomorphisms
$$\Pic X \cong (\Pic \overline{X})^{G_F}, \quad \HH^1(F,X^*(\Tbar)) \cong \Pic T, \quad \HH^1(F,\ZZ^{\overline{\A}})=0,$$
which follow from $X(F) \neq \emptyset$,
Lemma \ref{lem:Pic_Br_iso}, Shapiro's lemma and $\HH^1(F_\alpha,\ZZ)=0$.
Next fix an equivariant embedding $T \subset X$ and define $$\Br_e X  = \{ b \in \Br_1 X: b(1)=0\}.$$
Then the long exact sequence associated to (\ref{eqn:fundamental}) continues as
\begin{equation} \label{seq:purity_tori}
	0 \to \Br_e X \to \Br_e T \to \bigoplus_{\alpha \in \A} \Br_e T_\alpha,
\end{equation}
where here we have used the isomorphism $\Br_e T \cong  \HH^2(F,X^*(\Tbar))$
of Lemma \ref{lem:Pic_Br_iso} and the isomorphism $\Br_e X \cong \Br_a X \cong \HH^1(F, \Pic \Xbar)$ of \cite[Lem.~6.3(iii)]{San81}.

\subsubsection{Weak approximation}
A result due to Voskresenski\v{\i} \cite{Vos70} (see also \cite[Prop.~2.34]{Bou11}) implies that the sequence
(\ref{eqn:fundamental_dual}) gives rise to an exact sequence
\begin{equation}\label{eqn:toric_WA}
	\prod_{\alpha \in \A} T_{\alpha}(\Adele_F)/T_\alpha(F) \to T(\Adele_F)/T(F) \to A(T) \to  0,
\end{equation}
where $A(T)$ is given by (\ref{def:A(T)}).
From this we obtain the following exact sequence
\begin{equation} \label{eqn:toric_char_WA}
	0 \to \Fourier{A(T)} \to \Fourier{(T(\Adele_F)/T(F))} \to \prod_{\alpha \in \A} \Fourier{(T_{\alpha}(\Adele_F)/T_\alpha(F))}.
\end{equation}

\subsubsection{Purity}
The Grothendieck purity sequence (\ref{seq:purity0}) here reads
\begin{equation}\label{eqn:toric_purity}
    0 \to \Br_1 X \to \Br_1 T \to \bigoplus_{\alpha \in \A}\HH^1(F_\alpha,\QQ/\ZZ).
\end{equation}

Note that the sequences (\ref{seq:purity_tori}), (\ref{eqn:toric_char_WA}) and (\ref{eqn:toric_purity})
formally resemble  each other.
The following lemma shows that these sequences are indeed compatible.

\begin{lemma} \label{lem:toric_purity}
	The Brauer pairing yields a commutative diagram
	$$
	\xymatrix{
		0 \ar[r] & \Be(T) \ar[r] \ar[d] & \Br_e T \ar[r] \ar[d] & (T(\Adele_F)/T(F))^\sim \ar[d] \ar[r] & 0 \\
	 	& 0 \ar[r] & \bigoplus_{\alpha \in \A} \Br_e T_\alpha \ar[r]  &\prod_{\alpha \in \A} (T_{\alpha}(\Adele_F)/T_\alpha(F))^\sim \ar[r] & 0}
	$$
	with exact rows. Moreover,  the diagram
	$$
	\xymatrix{
		 \Br_e T \ar[d] \ar[r] & \bigoplus_{\alpha \in \A} \Br_e T_\alpha \ar[d] \\
		 \Br_1 T \ar[r] & \bigoplus_{\alpha \in \A} \HH^1(F_\alpha,\QQ/\ZZ) \\}
	$$
	commutes up to sign (the maps are explained in the proof).
\end{lemma}
\begin{proof}
	First note that each $T_\alpha$ is rational, being an open
	subset of affine space. Hence the first commutative diagram follows 
	from the functoriality of the Brauer pairing, Theorem \ref{thm:Brauer_global}
	and Corollary \ref{cor:rational}.
	In the second diagram, the top and bottom rows come from the exact sequences
	 (\ref{seq:purity_tori}) and (\ref{eqn:toric_purity}) respectively.
	The maps between them are the obvious inclusion combined with the
	isomorphism $\Br_e T_\alpha \cong \HH^2(F,X^*(T_\alpha))$ of Lemma
	\ref{lem:Pic_Br_iso} and the isomorphism $\HH^2(F,X^*(T_\alpha)) \cong \HH^1(F_\alpha,\QQ/\ZZ)$
	given by Shapiro's lemma and the definition of $T_\alpha$. The
	fact that this diagram commutes up to sign is proven in \cite[Lem.~9.1]{San81}.
\end{proof}

\subsection{Subordinate Brauer group elements on algebraic tori} \label{sec:sub_tori}
In this section we study subordinate Brauer group elements on algebraic tori
(see \S \ref{Sec:subordinate} for definitions).
Let $T$ be an algebraic torus over a number field $F$ and let $X$ be a toric
variety for $T$ together with a choice of an equivariant embedding $T \subset X$.
Fix a finite subset $\br \subset \Br T$. Let $\Sub_e(X,\br) = \Br_e T \cap \Sub(X,\br)$ and let
$\Sub_e(F(T),\br)=\Br_e T \cap \Sub(F(T),\br)$. 

Our first result is an analogue for algebraic subordinate Brauer group elements 
of the fact that the Brauer-Manin obstruction is the only obstruction to
the Hasse principle and  weak approximation for $T$.
To state this, let $T(F)_\br$ denote the zero-locus of $\br$ (see \eqref{eqn:Brauer_kernel}), and let
$\overline{T(F)}^w_\br$ denote the closure of $T(F)_\br$ in $T(\Adele_F)_\br$
with respect to the product topology.

\begin{theorem} \label{thm:sub_WA}
	Assume that $\br \subset \Br_1 T$. Then 	
    \begin{equation} \label{eqn:closure_br}
        T(\Adele_F)^{\Sub(F(T),\br)}_\br=\overline{T(F)}^w_{\br}.
    \end{equation}
	Moreover, there exists a finite set of places $S$ of $F$ such that
	$$\overline{T(F)}^w_{\br}=\left(\overline{T(F)}^w_{\br} \bigcap 
	\prod_{v \in S} T(F_v)_\br\right) \times \left(T(\Adele_F)_\br \bigcap \prod_{v \not \in S} T(F_v)_\br\right).$$
\end{theorem}
\begin{proof}
	We shall prove the result using Theorem \ref{thm:Brauer_Manin} and
	the tools developed in \S\ref{Sec:subordinate}.
	First note that by Theorem \ref{thm:Sub_Harari}, the pairing
	$$T(\Adele_F)_\br \times \Sub(F(T),\br) \to \QQ/\ZZ,$$
	is continuous with respect to the product topology. Hence
	\begin{equation} \label{eqn:subset}
		\overline{T(F)}^w_{\br} \subset T(\Adele_F)^{\Sub(F(T),\br)}_\br.
	\end{equation}
	Let $\pi:V \to T$ be the product of Brauer-Severi schemes associated to $\br$.
	Note that 
	\begin{equation} \label{eqn:closure}
		\pi(V(F)) = T(F)_\br \quad \text{and} \quad \pi(V(\Adele_F)) = T(\Adele_F)_\br.
	\end{equation}
	Moreover, by Theorem \ref{thm:CTSD} we know that the pull-back of $\Sub(F(T),\br)$
    to $V$ is exactly $\Brnr V$. Hence the funtoriality of the Brauer
    pairing and \eqref{eqn:closure} implies that 
    \begin{equation} \label{eqn:subset2}
        \pi(V(\Adele_F)^{\Brnr V}) = T(\Adele_F)^{\Sub(F(T),\br)}_\br.
    \end{equation}
    Next, by continuity and \eqref{eqn:closure}
    we have $\pi(\overline{V(F)}^w) \subset \overline{T(F)}^w_\br$.
    However, by Theorem \ref{thm:Brauer_Manin}, we know that
    $\overline{V(F)}^w = V(\Adele_F)^{\Brnr V}$. Combining
    this with \eqref{eqn:subset2} and \eqref{eqn:subset}
    proves the first part of the lemma.
	The second part of the lemma then follows from (\ref{eqn:closure_br}), 
	Theorem~\ref{thm:Sub_Harari} and Lemma \ref{lem:Sub_is_finite}.
\end{proof}

Assume now that $\br \subset \Br_e T$.
Note that in the classical case where $\br =0$, the group $A(T)$ 
has an alternative description given by (\ref{eqn:toric_char_WA}).
This description was crucial in \cite{BT95} when calculating the leading
constant in the asymptotic formula (see the proof of \cite[Thm.~3.4.6]{BT95}).
For subordinate
Brauer group elements, however, the analogous 
group $A(T, \br) = T(\Adele_F)_{\br}/\overline{T(F)}^w_{\br}$ 
does not play this r\^{o}le. To state
the result let $\R \subset (T(\Adele_F)/T(F))^{\sim}$
denote the collection of characters obtained from $\br$ via Theorem \ref{thm:Brauer_global}.
Let $\br_\alpha \subset \Br_e T_\alpha$ and $\R_\alpha \subset (T_\alpha(\Adele_F)/T_\alpha(F))^{\sim}$
for $\alpha \in \A$ denote the corresponding elements obtained via pull-back in the sequence
(\ref{eqn:fundamental_dual}).

\begin{theorem} \label{thm:sub_C}
    Assume that $\br \subset \Br_e T$. Let
    $$C(T,\R) = \{ \chi \in \Fourier{(T(\Adele_F)/T(F))} : \chi_\alpha \in \R_\alpha \text{ for all } \alpha \in \A\}.$$
	Then $C(T,\R)$ is finite and the Brauer pairing induces a short exact sequence
    $$0 \to \Be(T) \to \Sub_e(X,\br) \to C(T,\R) \to 0.$$
\end{theorem}
\begin{proof}
	First note that, by \eqref{eqn:toric_WA}, the image of the natural map
	$$\prod_{\alpha \in \A} T_{\alpha}(\Adele_F)/T_\alpha(F) \to T(\Adele_F)/T(F)$$
	is closed of finite index. As each $\R_\alpha \subset (T_\alpha(\Adele_F)/T_\alpha(F))^{\sim}$,
	we deduce that $C(T,\R) \subset (T(\Adele_F)/T(F))^\sim$.
    Therefore, by Theorem \ref{thm:Brauer_global}, we know that $C(T,\R)$
	lies in the image of the map $\epsilon: \Br_e T \to (T(\Adele_F)/T(F))^{\sim}$. We claim
	that $\epsilon$ restricts to a surjection $\Sub_e(X,\br) \to C(T,\R)$.
	Let $b \in \Br_e T$ with corresponding character $\chi \in (T(\Adele_F)/T(F))^{\sim}$.
    Then Lemma \ref{lem:toric_purity} implies that
	\begin{align*}
		b \in \Sub_e(X,\br) &\iff \res_{D_\alpha}(b) \in \langle \res_{D_\alpha}(\br) \rangle \text{ for all } \alpha \in \A, \\
		&\iff b_\alpha \in \br_\alpha \text{ for all } \alpha \in \A, \\
		&\iff \chi_\alpha \in \R_\alpha \text{ for all } \alpha \in \A, \\
		& \iff \chi \in C(T,\R),
	\end{align*}
	thus proving the claim. The finiteness of $C(T,\R)$ therefore follows from Lemma \ref{lem:Sub_is_finite}.
	The equality $\ker(\Sub_e(X,\br) \to C(T,\R)) = \Be(T)$ now follows from Theorem \ref{thm:Brauer_global}.
\end{proof}
Note that we recover (\ref{seq:Br_Sha}) from Theorem \ref{thm:sub_C} on taking $\br=0$
and using (\ref{eqn:toric_char_WA}).

\subsection{Heights} \label{sec:Heights}
Let $X$ be a projective variety over a number field $F$ and let $L$ be a line bundle on $X$. For a place $v \in \Val(F)$,
a \emph{$v$-adic metric} $|| \cdot ||_v$ on $L$ is a continuously varying family of $v$-adic norms
on the fibres of $L$. An \emph{adelic metric} on $L$ is a collection $|| \cdot||=(||\cdot||_v)$ of
$v$-adic metrics on $L$ for each place $v \in \Val(F)$, such that almost all of these $v$-adic metrics are defined
by a fixed model of $X$ over $\OO_F$ (see e.g.~\cite{CLT10}).
We call the data $\LL=(L,||\cdot||)$ an adelically metrised line bundle.
Given a rational point $x \in X(F)$, we define the height of $x$ with respect to $\LL$ to be
$$H_{\LL}(x)=\prod_{v \in \Val(F)} ||\ell(x)||_v^{-1},$$
where $\ell$ is any local section of $L$ defined at $x$ such that $\ell(x) \neq 0$.

\subsubsection{Heights on toric varieties}
Now suppose that $X$ is a toric variety for an algebraic torus $T$ with a fixed
equivariant embedding $T \subset X$. We shall extend the previous
construction in two ways; namely we define the ``height'' of an adelic point and
form a system of complex  height functions on every line bundle of $T$
in a compatible manner. Choose adelic metrics $||\cdot||_\alpha$ on the line bundles $\OO_X(D_\alpha)$
for each $\alpha \in \A$ (see \S \ref{sec:toric_varieties}). Let $d_\alpha$ denote the
global section of $\OO_X(D_\alpha)$ corresponding to the divisor $D_\alpha$.
We then define the following local height pairing
\begin{align*}
	&H_v: T(F_v) \times \CC^\A \to \CC^*, \quad
	 ( t_v; \sbf) \mapsto \prod_{\alpha \in \A}||d_\alpha (t_v)||_{\alpha,v}^{-s_\alpha},
\end{align*}
and also the following global height pairing
\begin{align*}
	&H: T(\Adele_F) \times \CC^\A \to \CC^*, \quad
	( (t_v); \sbf) \mapsto \prod_{v \in \Val(F)} H_v((t_v); \sbf).
\end{align*}
Here we write $\sbf = (s_\alpha)_{\alpha \in \A}$.
As an example, from (\ref{eqn:canonical_bundle})  the adelic metrics on
$\OO_X(D_\alpha)$ induce an adelic metric on $\omega_X^{-1}$, and on taking $s_\alpha=s$ for each $\alpha \in \A$
we have
$$H(t; (s)_{\alpha \in \A})=H_{\omega_X^{-1}}(t)^s,$$
for each $t \in T(F)$. Batyrev and Tschinkel \cite[Def.~2.1.5]{BT95}
constructed ``canonical'' adelic metrics on the line bundles
$\OO_X(D_\alpha)$. In this paper we shall focus on these adelic metrics, as it greatly simplifies the harmonic
analysis in the proof of Theorem \ref{thm:Brauer}. The key property of these metrics is that
the corresponding local height functions are $T(\OO_v)$-invariant and trivial on $T(\OO_v)$
for every place $v$ of $F$ \cite[Thm.~2.16]{BT95}.

\section{Counting functions associated to Brauer group elements} \label{sec:counting}
\subsection{The set-up}
We now begin the proof of Theorem \ref{thm:Brauer}. Throughout this section $X$ is a toric variety over a
number field $F$ with respect to an algebraic torus $T$, with set of boundary components $\A$
(see \S \ref{sec:toric_varieties}).
We only specialise to the case where $T$ is anisotropic just before we apply the
Poisson summation formula in \S \ref{Sec:anisotropic_case}.
We fix a finite subgroup $\br \subset \Br_1 U$ of the open dense orbit $U \subset X$ such that $U(F)_\br \neq \emptyset$.
We fix a choice of equivariant embedding $T \subset X$ such that $1 \in U(F)_\br$.
In particular we may identify $\br$ with a finite subgroup of $ \Br_e T$ (see (\ref{def:Br_e})).
We also equip the line bundles $\OO_X(D_\alpha)$ with the canonical Batyrev-Tschinkel adelic
metrics (see \S \ref{sec:Heights}).

Let $\thorn:T(F) \to \{0,1\}$ be the indicator function of the zero-locus $T(F)_\br$ of $\br$.
This extends to a locally constant function (also denoted $\thorn$) on $T(\Adele_F)$,
given as a product of local indicator functions $\thorn_v$ (see \S\ref{sec:Br_fields}).
Let $\R$ be the group of
automorphic characters of $T(\Adele_F)$ associated to $\br$ via Theorem \ref{thm:Brauer_global}.
By character orthogonality we have
\begin{equation} \label{eqn:thorn_rho}
	\thorn_{v} = \frac{1}{|\R|}\sum_{\rho \in \R} \rho_{v},
\end{equation}
as $\thorn_{v}$ is the indicator function for $\bigcap_{\rho \in \R} \ker \rho_v$.
Note that $\thorn$ is \emph{not} the indicator function of
$\bigcap_{\rho \in \R}\ker \rho \subset T(\Adele_F)$ in general; indeed
$T(F) \subset \bigcap_{\rho \in \R}\ker \rho$, as each $\rho$ is automorphic.

Using (\ref{eqn:fundamental_dual}), we may pull-back $\br$ to obtain a collection of Brauer group elements $\br_\alpha \in \Br_e T_\alpha$ for each $\alpha \in \A$.
Let $\thorn_\alpha:T_\alpha(\Adele_F) \to \{0,1\}$ denote the indicator function
of $T_\alpha(\Adele_F)_{\br_\alpha}$. We denote the residue map associated to the divisor $D_\alpha$
by $\res_\alpha$.
We may pull-back characters $\chi$ of $T(\Adele_F)$ to obtain characters $\chi_\alpha$ of
$T_\alpha(\Adele_F)$. For automorphic $\chi$, we identify these with Hecke characters of $F_\alpha$
using the identification $T_\alpha(\Adele_F)=\Adele^*_{F_\alpha}$.
We let $\R_\alpha$ denote the collection of characters on $T_{\alpha}(\Adele_F)$ induced by the $\br_{\alpha}$
via Theorem \ref{thm:Brauer_global}. These coincide with the pull-back of $\R$ to $T_\alpha(\Adele_F)$, 
as the Brauer pairing is functorial.
We identify those places of $F_\alpha$ which lie above a fixed non-archimedean place $v$ of $F$,
with those elements $\alpha_v \in \A_v$
such that $\alpha_v \mid \alpha$. With respect to the identification
$T_\alpha(\Adele_F)=\Adele^*_{F_\alpha}$, we may therefore write each
$\thorn_\alpha$ (resp.~$\chi_\alpha$) as a product of local factors $\thorn_{\alpha_v}$
(resp.~$\chi_{\alpha_v}$) for each place $v$ of $F$ and each $\alpha_v \mid \alpha$.

We fix a finite set $S \subset \Val(F)$ containing all archimedean places,
all places which are ramified in the splitting field of $T$,
all places $v$ for which $\rho_v$ is ramified for some $\rho \in \R$
and all places $v$ for which $\vol(\mathcal{O}_v) \neq 1$ with respect
to our Haar measure on $F_v$.

\subsection{The height zeta function and the Poisson summation formula}
Throughout $\sbf=(s_\alpha) \in \CC^{\A}$ denotes a complex variable.
Given $c \in \RR$, we say that $\re \sbf > c$
if $\sbf$ lies in the complex tube domain
$\{\sbf \in \CC^\A:\re s_\alpha > c \text{ for all } \alpha \in \A\}.$
The generating series for the rational points of interest is the following height zeta function
\begin{equation} \label{def:HZF}
	Z(\sbf) = \sum_{t \in T(F)} \thorn(t)H(t;-\sbf) = \sum_{t \in T(F)_\br} H(t;-\sbf).
\end{equation}
We shall often be able to reduce the study of this zeta function to the works of \cite{BT95} and \cite{BT98}.
For example since $|\thorn(t)| \leq 1$ for each $t \in T(F)$,
we deduce immediately from \cite[Thm.~4.2]{BT98} 
that the sum in (\ref{def:HZF}) converges absolutely on the tube domain $\re \sbf >1$,
and hence defines a holomorphic function on this domain.

The key tool in the study of $Z(\sbf)$ is the Poisson summation formula.
Let $f:T(\Adele_F) \to \CC$ be a continuous function which is given as a product of local factors $f_v$
such that $f_v(\OO_v)=1$ for almost all places $v$ of $F$.
We define the Fourier transform of a character $\chi \in \Fourier{T(\Adele_F)}$ with respect to $f$ to be
$$ \widehat{H}(f,\chi;-\sbf) = \int_{T(\Adele_F)} f(t) \chi(t) H(t;-\sbf) \mathrm{d} \mu,$$
for those $\sbf \in \CC^\A$ for which the integral exists.
Our assumptions on $f$ imply that the Fourier transform decomposes as a product of local Fourier transforms
$$ \widehat{H}(f,\chi;-\sbf)= \prod_{v \in \Val(F)}\widehat{H}_v(f_v,\chi_v;-\sbf),$$
where
$$\widehat{H}_v(f_v,\chi_v;-\sbf) = \int_{T(F_v)} f_v(t_v) \chi_v(t_v) H_v(t_v;-\sbf) \mathrm{d} \mu_v.$$
A formal application of the Poisson summation formula (see \cite[Thm.~4.4]{BT98}, \cite[\S3.5]{Bou11}) gives
\begin{equation} \label{eqn:Poisson}
	Z(\sbf)= \frac{1}{(2\pi)^{\rank X^*(T)}\vol(T(\Adele_F)^1/T(F))}
	\int_{\chi \in \Fourier{\left(T(\Adele_F)/T(F)\right)}}\widehat{H}(\thorn,\chi;-\sbf) \mathrm{d} \chi.
\end{equation}
Note that $T(\Adele_F) = T(\Adele_F)^1$ when $T$ is anisotropic by (\ref{seq:adelic_splitting}),
hence $T(\Adele_F)/T(F)$ is compact. Thus in this case $\Fourier{\left(T(\Adele_F)/T(F)\right)}$ is discrete and therefore the
above integral is really a \emph{sum} over characters.

\subsection{The local Fourier transforms}
We begin by studying the local Fourier transforms. By
(\ref{eqn:thorn_rho}), for any place $v$ of $F$  we have
\begin{equation} \label{eqn:Fourier_thorn_rho}
	\widehat{H}_v(\thorn_v,\chi_v;-\sbf) = \frac{1}{|\R|}\sum_{\rho \in \R} \widehat{H}_v(1,\rho_v\chi_v;-\sbf).
\end{equation}
We shall use this to reduce the study of the local Fourier transforms
to the work of Batyrev and Tschinkel \cite{BT95, BT98}, which corresponds
to the case  $\R = 1$.

\begin{lemma} \label{lem:local_Fourier}
	Let $v \in \Val(F)$ and let $\varepsilon >0$. 
	Let $\chi_v$ be a character of $T(F_v)$. Then the local Fourier transform
	$\widehat{H}_v(\thorn_v,\chi_v;-\sbf)$ converges absolutely and is uniformly bounded (in terms of $\varepsilon$ and $v$)
	on the tube domain $\re \sbf \geq 1/2+ \varepsilon$.
\end{lemma}
\begin{proof}
	By \eqref{eqn:Fourier_thorn_rho} it suffices to prove the corresponding result for 
	$\widehat{H}_v(1,\rho_v\chi_v;-\sbf)$. This is the content of \cite[Rem.~2.2.8]{BT95}
	and \cite[Prop.~2.3.2]{BT95}.
\end{proof}

We also show that the local Fourier transform of the trivial character does not vanish.

\begin{lemma}\label{lem:trivial_non_vanishing}
	Let $v \in \Val(F)$. Then	$\widehat{H}_v(\thorn_v,1;-\sbf)$
	is non-zero for any $\sbf \in \RR_{>1/2}^\A$.
\end{lemma}
\begin{proof}
	As $\thorn_v$ is locally constant with $\thorn_v(1)=1$ and $T(F_v)$ is locally compact, there exists a
	compact neighbourhood $C_v$ of $1$ such that $\thorn_v(C_v) =1$.
	For $\sbf \in \RR_{>1/2}^\A$, we have
	\begin{equation*}
		\widehat{H}_v(\thorn_v,1;-\sbf) \geq \int_{C_v} H_v(t_v;-\sbf) \mathrm{d} \mu_v > 0. \qedhere
	\end{equation*}
\end{proof}

\subsubsection{The non-archimedean places}
We now show that our height functions and $\thorn$ are well-behaved at the non-archimedean places
from a harmonic analysis perspective.
\begin{lemma}\label{lem:thorn_invariant_finite}
	Let $v \in \Val(F)$ be non-archimedean. Let $K_v \subset T(\OO_v)$ be
	the maximal subgroup for which $H_v(\cdot;-\sbf)$ and $\thorn_v$ are both
	$K_v$-invariant and trivial on $K_v$. Then $K_v$ is compact, open
	and of finite index.
    Moreover when $v \not \in S$, one has $K_v= T(\OO_{v})$.
\end{lemma}
\begin{proof}
	The result for $\thorn_v$ follows from (\ref{eqn:thorn_rho}).
	The result for $H_v(\cdot;-\sbf)$ follows from 
	our choice of height functions (see \S \ref{sec:Heights}).	
\end{proof}

\begin{lemma} \label{lem:finite_ramified_vanishing}
	Let $v \in \Val(F)$ be non-archimedean and let $\chi_v$ be a character of $T(F_v)$ which is non-trivial
	on $K_v$. Then
	$$\widehat{H}_v(\thorn_v,\chi_v;-\sbf)=0.$$
\end{lemma}
\begin{proof}
	Applying character orthogonality, we obtain
	$$
		\widehat{H}_v(\thorn_v,\chi_v;-\sbf) 
		=\sum_{n_v \in T(F_v)/K_v}\thorn_v(n_v) \chi_v(n_v)  H_v(n_v;-\sbf) \int_{K_v} \chi_v(t_v) \mathrm{d} \mu_v
		=0,
	$$
	as required.
\end{proof}

\begin{remark}
	Lemma \ref{lem:finite_ramified_vanishing} 
	plays a key r\^{o}le in showing the convergence of the Poisson integral.
	This is an important part of the paper where we use the fact that our Brauer group elements
	are \emph{algebraic}, as the conclusion of Lemma \ref{lem:thorn_invariant_finite}
	does not hold for \emph{transcendental} Brauer group elements in general. 
	For example, consider the quaternion algebra $(x,y)$ on $\Gm^2 \subset \mathbb{A}^2$ over $\QQ$.
	A simple Hilbert symbol computation shows that for $p >2$, the corresponding indicator function is
	$(1 + p\ZZ_p)^2$-invariant, but not $(\ZZ_p^* \times \ZZ_p^*)$-invariant.
\end{remark}

We now obtain more explicit information about the local Fourier transforms for places not in $S$. In what follows,
for $\sbf \in \CC^\A$ and $\alpha_v \in \A_v$, the symbol $s_{\alpha_v}$ denotes the complex number $s_\alpha$, where $\alpha \in \A$
is the unique element such that $\alpha_v \mid \alpha$ (see \S \ref{sec:toric_varieties}).
\begin{lemma} \label{lem:good_Fourier}
	Let $\varepsilon >0$, let $v \not \in S$ be a place of $F$ and
	let $\chi_v$ be an unramified character of $T(F_v)$.
	Then on the tube domain $\re \sbf \geq 1/2 + \varepsilon$ we have
	$$\widehat{H}_v(\thorn_v,\chi_v;-\sbf)=
	\prod_{\alpha_v \in \A_v}\left(1-\frac{\thorn_{\alpha_v}(\pi_{\alpha_v})\chi_{\alpha_v}(\pi_{\alpha_v})}{q_v^{f_{\alpha_v}s_{\alpha_v}}}\right)^{-1}
	\left(1 + O_\varepsilon\left(\frac{1}{q_v^{1+\varepsilon}}\right)\right).$$
\end{lemma}
\begin{proof}
	The proof of \cite[Thm.~3.1.3]{BT95} shows that
	$$\widehat{H}_v(1,\chi_v;-\sbf)=\prod_{\alpha_v \in \A_v}
	\left(1-\frac{\chi_{\alpha_v}(\pi_{\alpha_v})}{q_v^{f_{\alpha_v}s_{\alpha_v}}}\right)^{-1}
	\left(1 + O_\varepsilon\left(\frac{1}{q_v^{1+\varepsilon}}\right)\right).$$
	Combining this with (\ref{eqn:Fourier_thorn_rho}) we obtain
	\begin{align*}
		\widehat{H}_v(\thorn_v,\chi_v;-\sbf)
		&= 1+\frac{1}{|\R|}\sum_{\rho \in \R}\sum_{\alpha_v \in \A_v}
		\frac{\rho_{\alpha_v}(\pi_{\alpha_v})\chi_{\alpha_v}(\pi_{\alpha_v})}{q_v^{f_{\alpha_v}s_{\alpha_v}}}+
		O_\varepsilon\left(\frac{1}{q_v^{1+\varepsilon}}\right).
	\end{align*}
	The result then follows from character orthogonality, which implies that
	\begin{equation*} 
		\frac{1}{|\R|} \sum_{\rho \in \R}\rho_{\alpha_v}(\pi_{\alpha_v}) = \thorn_{\alpha_v}(\pi_v). \qedhere
	\end{equation*}
\end{proof}

\subsubsection{The archimedean places}
If $v \in \Val(F)$ is archimedean, then $\thorn_v$ is very easy to describe.
Namely, if $v$ is complex then $\thorn_v = 1$ as $\Br \CC = 0$. If $v$ is real,
then it is well-known that we have an isomorphism
$$
	T(F_v) \cong (\RR^{*})^{r_1} \times (\RR_{>0})^{r_2} \times (S^1)^{r_3},
$$
of topological groups, for some $r_1,r_2,r_3 \geq 0$. On noting that $\thorn_v$ is locally
constant and using (\ref{eqn:thorn_rho}), we see that $\thorn_v$ is simply the indicator function
of an open and closed subgroup of $T(F_v)$ whose index divides $2^{r_1}$.
These remarks easily allow us to prove the archimedean analogues of Lemma \ref{lem:thorn_invariant_finite}
and Lemma \ref{lem:finite_ramified_vanishing}.
\begin{lemma}\label{lem:thorn_invariant_infinite}
	Let $v \in \Val(F)$ be archimedean and let $K_v \subset T(\OO_v)$ be
	the maximal subgroup for which $H_v(\cdot;-\sbf)$ and $\thorn_v$ are both
	$K_v$-invariant and trivial on $K_v$. Then $K_v$ is compact, open and
	of finite index.
\end{lemma}
\begin{proof}
	The required property for $H_v(\cdot;-\sbf)$ follows as we are using the Batyrev-Tschinkel height
	(see \S \ref{sec:Heights}). The statement	for $\thorn_v$ follows from the above remarks.
\end{proof}

\begin{lemma}\label{lem:infinite_ramified_vanishing}
	Let $v \in \Val(F)$ be archimedean and let $\chi_v$ be a character of $T(F_v)$ which is non-trivial on $K_v$.
	Then
	$$\widehat{H}_v(\thorn_v,\chi_v;-\sbf)=0,$$
\end{lemma}
\begin{proof}
	Equip $K_v$ with a Haar measure $\kappa_v$ and denote by $\overline{\mu}_v$
	the induced quotient measure on $T(F_v) / K_v$.	Then we have
	\begin{equation*}
		\widehat{H}_v(\thorn_v,\chi_v;-\sbf) 
		=\int_{n_v \in T(F_v)/K_v}\thorn_v(n_v) \chi_v(n_v)  H_v(n_v;-\sbf) \mathrm{d} \overline{\mu}_v 
		\int_{K_v} \chi_v(t_v) \mathrm{d} \kappa_v =0. \qedhere
	\end{equation*}
\end{proof}

The following will be used to handle the sum appearing in the
Poisson formula (\ref{eqn:Poisson}).

\begin{lemma}\label{lem:archimedean_sum}
	Choose a $\RR$-vector space norm $||\cdot||$
	on $X^*(T_\infty)_\RR$ and let $\mathcal{L} \subset X^*(T_\infty)_\RR$ be a lattice.
	Let $C$ be a compact subset of $\re \sbf \geq 1$ and let
	$g:X^*(T_\infty)_\RR \times C \to \CC$ be a function. Suppose there exists
	some $0\leq \delta<1/\dim X$ such that
	$$|g(\psi,\sbf)| \ll_C (1 + ||\psi||)^{\delta},$$
	for all $\psi \in  X^*(T_\infty)_\RR$ and all $\sbf \in C$.
	Then the sum
	$$\sum_{\psi \in \mathcal{L}}g(\psi,\sbf) \prod_{v \mid \infty} \widehat{H}_v(\thorn_v,\psi_v; -\sbf),$$
	is absolutely and uniformly convergent on $C$, where we write $\psi=(\psi_v)_{v \mid \infty}$.
\end{lemma}
\begin{proof}
	By (\ref{eqn:Fourier_thorn_rho}) we obtain similar bounds for $\widehat{H}_v(\thorn_v,\psi_v; -\sbf)$
	to those given in \cite[Prop.~2.3.2]{BT95}. Hence the result follows from \cite[Cor.~2.3.4]{BT95}.
\end{proof}



\subsection{The global Fourier transform}
We now move onto the global Fourier transforms. We shall relate
these to the partial Euler products considered in \S \ref{sec:partial_zeta_functions}.

\begin{lemma} \label{lem:global_fourier}
    Let $\varepsilon >0$ and let $\chi$ be an automorphic character of $T(\Adele_F)$.
	Then there exists a function $\varphi(\chi;\sbf)$	which is holomorphic and uniformly bounded with respect to
	$\chi$ on $\re \sbf \geq 1/2 + \varepsilon$, such that
	$$\widehat{H}(\thorn,\chi;-\sbf)=\prod_{v \mid \infty}\widehat{H}_v(\thorn_v,\chi_v; -\sbf)
	 \prod_{\alpha \in \A} L_{\R_\alpha}(\chi_{\alpha},s_{\alpha})\varphi(\chi;\sbf),$$
	for $\re \sbf > 1$.
\end{lemma}
\begin{proof}
	Recalling the definition of the partial Euler product (\ref{def:partial_zeta_function}), this follows from Lemma \ref{lem:local_Fourier},
	Lemma \ref{lem:good_Fourier} and the fact that for non-archimedean places $v$, we have identified those
	elements $\alpha_v \mid \alpha$ with those places $w$ of $F_\alpha$ such that $w\mid v$.	
\end{proof}

\subsubsection{The anisotropic case}\label{Sec:anisotropic_case}
So far, all our arguments have applied to arbitrary
tori.
$$\text{From now on we assume that } T \text{ is anisotropic}.$$
We also restrict to the complex line $\sbf= (s)_{\alpha \in \A}$ in $\CC^\A$.
As explained in \S \ref{sec:Heights}, the resulting
height function $H(\cdot;-s)$ is a complex power of the anticanonical
height function.

\begin{lemma} \label{lem:global_fourier_anisotropic}
	Let $\chi$ be an automorphic
	character of $T(\Adele_F)$. Then $\widehat{H}(\thorn,\chi;-s)$
	admits a holomorphic continuation to the line $\re s=1$, apart
	from possibly at $s=1$. Here
	$$\widehat{H}(\thorn,\chi;-s) = \prod_{\substack{\alpha \in \A \\ \chi_\alpha \in \R_\alpha}}
    \left( \frac{c_{\R,\chi,\alpha}}{(s-1)^{1/|\R_\alpha|}} + O\left(\frac{1}{(s-1)^{1/|\R_\alpha|-1}}\right)\right),$$
	as $s \to 1$, for some constants $c_{\R,\chi,\alpha}$ which are non-zero if $\chi =1$.
\end{lemma}
\begin{proof}
    By Lemma \ref{lem:local_Fourier} and Lemma \ref{lem:global_fourier},
    it suffices to study the partial
    Euler products $L_{\R_\alpha}(\chi_{\alpha},s)$.
    By Lemma \ref{lem:partial_Euler_product}, we know that $L_{\R_\alpha}(\chi_{\alpha},s)$
    has a singularity on the line $\re s =1$ if and only if $\chi_\alpha = m_\alpha \rho_\alpha$ for some
    $\rho_\alpha \in \R_\alpha$ and some $m_\alpha \in X^*(T_\alpha)_\RR$. As $T$ is anisotropic however
    we have $X^*(T) = 0$. It follows from the functorial isomorphism (\ref{eqn:adelic_splitting})
    that for any such character we have $m_\alpha = 0$, in particular $\chi_\alpha \in \R_\alpha$.
    Hence Lemma \ref{lem:partial_Euler_product} implies that the singularity of $L_{\R_\alpha}(\chi_{\alpha},s)$
    occurs at $s=1$. The non-vanishing of the constants
    $c_{\R,1,\alpha}$ follows from Lemma \ref{lem:partial_Euler_product}
    and Lemma \ref{lem:trivial_non_vanishing}.
\end{proof}

\subsection{The asymptotic formula}
We now apply the Poisson formula (\ref{eqn:Poisson}) to obtain the following.

\begin{theorem}\label{thm:height_zeta_function}
    Let $$\Omega(s)=Z(s)(s-1)^{\sum\limits_{\,\alpha \in \A} 1/|\R_\alpha|}.$$
	Then $\Omega(s)$ admits an extension to an infinitely differentiable function on $\re s \geq 1$.
    Moreover we have
    $$Z(s)= \Omega(1) (s-1)^{-\sum\limits_{\mathclap{\alpha \in \A}} 1/|\R_\alpha|}
    + O\left((s-1)^{1 - \sum\limits_{\mathclap{\alpha \in \A}} 1/|\R_\alpha|} \right), 
    \quad \text{as } s \to 1. $$
\end{theorem}
\begin{proof}
	As $T$ is anisotropic, the Poisson formula (\ref{eqn:Poisson}) reads
	\begin{equation} \label{eqn:HZF_anisotropic}
    	Z(s)= \frac{1}{\vol(T(\Adele_F)/T(F))}
    	\sum_{\chi \in \Fourier{\left(T(\Adele_F)/T(F)\right)}}\widehat{H}(\thorn,\chi;-s).
	\end{equation}
	To show that the application of the Poisson formula is valid, we shall use the criterion given by Bourqui
	\cite[Cor.~3.36]{Bou11}. By \emph{loc.~cit.} it suffices to show that
	the sum in \eqref{eqn:HZF_anisotropic} is absolutely convergent for $\re s > 1$,
	and that there exists an open neighbourhood of the origin $\Omega \subset T(\Adele_F)$ 
	and strictly positive constants $C_1$ and $C_2$ such that for all $\omega \in \Omega$
	and all $t \in T(\Adele_F)$ we have
	\begin{equation} \label{eqn:Bourqui}
		C_1|\thorn(t)H(t;-s)| \leq |\thorn(\omega t)H(\omega t;-s)| \leq C_2|\thorn(t)H(t;-s)|.
	\end{equation}
	We take $\Omega_v = K_v$ for $v$ non-archimedean, so that $\thorn(\omega_vt_v)H(\omega_vt_v; -s) 
	= \thorn(t_v)H(t_v; -s)$ for all $\omega_v \in \Omega_v$ and all $t_v \in T(F_v)$,
	by Lemma \ref{lem:thorn_invariant_finite}. For archimedean $v$, we take $\Omega_v$
	to be the intersection of $T(F_v)_{\br}$ with the inverse image in 
	$T(F_v)$ of an open ball in $X_*(T_v)$. By \eqref{eqn:thorn_rho} we see that $\Omega_v$ is open and that 
	$\thorn(\omega_vt_v) = \thorn(t_v)$
	for all $\omega_v \in \Omega_v$ and all $t_v \in T(F_v)$. Hence on taking $\Omega=\prod_{v \in \Val(F)} \Omega_v$,
	the claim \eqref{eqn:Bourqui} follows as in \cite[Lem.~3.22, Lem.~3.23]{Bou11}.
	
	For absolute convergence on $\re s > 1$, let $A = \sum_{\alpha \in \A} 1/|\R_\alpha|$.
	We will show the stronger claim that the sum 
    \begin{equation} \label{eqn:to_show}
    \sum_{\chi \in \Fourier{\left(T(\Adele_F)/T(F)\right)}}
    \widehat{H}(\thorn,\chi;-s)(s-1)^A,
    \end{equation}
    is absolutely and uniformly convergent on any compact subset $C$ of the half-plane $\re s \geq 1$.
    This will show the validity of \eqref{eqn:HZF_anisotropic} and, on applying 
    Lemma \ref{lem:global_fourier_anisotropic},
    also complete the proof of the theorem.
    
	Let $K=\prod_{v \in \Val(F)}K_v$ and denote by $\mathcal{U}$ the group of automorphic characters
	of $T$ which are trivial on $K$. Note that by Lemma \ref{lem:finite_ramified_vanishing} and 
	Lemma \ref{lem:infinite_ramified_vanishing},
	the sum in (\ref{eqn:to_show}) may be taken only over those characters which lie in $\mathcal{U}$.
	As $K \subset K_T$ is a subgroup of finite index, it follows that the type at infinity
	map (\ref{def:type_at_infinity}) yields a homomorphism
	\begin{align*}
		\mathcal{U} &\to   X^*(T_\infty)_\RR, \quad
		\chi \mapsto \chi_\infty,
	\end{align*}
	which has finite kernel $\mathcal{K}$ and whose image $\mathcal{L}$ is a lattice of full rank.
	We obtain
	$$
		\sum_{\chi \in \mathcal{U}}\widehat{H}(\thorn,\chi;-s)(s-1)^A
		= \sum_{\psi \in \mathcal{L}} \prod_{v \mid \infty}\widehat{H}_v(\thorn_v,\psi_v; -s)
		\sum_{\substack{\chi \in \mathcal{U}\\ \chi_\infty = \psi}} \prod_{v \nmid \infty} \widehat{H}_v(\thorn_v,\chi_v;-s)(s-1)^A.
	$$
	Therefore by Lemma \ref{lem:global_fourier}, for $s \in C$ we have
	\begin{align*}
		\sum_{\chi \in \mathcal{U}} \widehat{H}(\thorn,\chi;-s)(s-1)^A
		& \ll \sum_{\psi \in \mathcal{L}} \prod_{v \mid \infty} |\widehat{H}_v(\thorn_v,\psi_v; -s)|
		\sum_{\substack{\chi \in \mathcal{U}\\ \chi_\infty = \psi}} \prod_{\alpha \in \A} |L_{\R_\alpha}(\chi_{\alpha},s)(s-1)^A|.
	\end{align*}
	As $K \subset K_T$ has finite index, there exists  $Q >0$ such that $q(\chi_\alpha) < Q$ 
	for all $\chi \in \mathcal{U}$
	and all $\alpha \in \A$. Therefore Lemma \ref{lem:Hecke_bound},
	Lemma \ref{lem:partial_Euler_product} and the finiteness of $\mathcal{K}$ imply that
	for any $\varepsilon >0$ and $s \in C$ we have
	$$\sum_{\substack{\chi \in \mathcal{U}\\ \chi_\infty = \psi}} 
	\prod_{\alpha \in \A} |L_{\R_\alpha}(\chi_{\alpha},s)(s-1)^A|
	\ll_{\varepsilon,C} |\mathcal{K}| \cdot Q^{\varepsilon} \cdot 
	(1 + \max_{\alpha \in \A}||\psi_\alpha||)^{\varepsilon}, $$
	for each $\psi \in \mathcal{L}$. Here $(\psi_\alpha)$ is the image of $\psi$ under the map
	$X^*(T_\infty) \to \prod_{\alpha \in \A} X^*(T_{\alpha,\infty})$
	and $||\psi_\alpha||$ is defined as in (\ref{def:character_norm}).
	The result therefore follows from Lemma \ref{lem:archimedean_sum}.
\end{proof}

\begin{remark}\label{rem:continuation}
	We are unable to prove that $Z(s)$ admits a \emph{holomorphic}
	extension to the line $\re s= 1$, away from $s=1$. To get such a result one requires
	uniform zero-free regions for Hecke $L$-functions, in order
	to obtain uniform holomorphic continuations of the partial
	Euler products. The generalised Riemann hypothesis would show that $Z(s)$ admits a holomorphic
	continuation to a half-plane $\re s > 1-\delta$, away from the branch cut at $s=1$.
	Unfortunately, current zero-free regions for Hecke $L$-functions
	(see \cite[Thm.~5.10]{IK04}) approach the line $\re s =1$ as one varies the infinity
	type of the character. Hence it does not seem possible with current technology
	to obtain any kind of continuation of $Z(s)$ in the half-plane $\re s <1$.
\end{remark}
In order to apply Theorem \ref{thm:Delange} to deduce Theorem \ref{thm:Brauer},
we need to know that
\begin{equation} \label{eqn:non-zero}
    \Omega(1)=\lim_{s \to 1} Z(s)(s-1)^{\sum\limits_{\,\alpha \in \A} 1/|\R_\alpha|} \neq 0.
\end{equation}
It should be emphasised that (\ref{eqn:non-zero}) does not follow from
what we have shown so far; since more than one character may give
rise to the singularity of highest order, it is theoretically possible
that cancellation may occur. We postpone the proof of (\ref{eqn:non-zero}) for now.
As $T$ is anisotropic, by (\ref{seq:toric_Pic}) we have $\#\A = \rho(X)$.
Lemma \ref{lem:toric_purity} also implies that
$$|\R_\alpha| = |\res_\alpha(\br)|, \quad \text{for all } \alpha \in \A.$$
Hence we deduce that
\begin{equation} \label{eqn:rho_Delta}
	\sum_{\alpha \in \A} \frac{1}{|\R_\alpha|} = \rho(X) - \Delta_X(\br),
\end{equation}
where $\Delta_X(\br)$ is as in Theorem \ref{thm:Brauer}.
Therefore assuming (\ref{eqn:non-zero}), we may apply Theorem \ref{thm:Delange}
and use Theorem \ref{thm:height_zeta_function} to find that
\begin{equation} \label{eqn:asym}
	N(T,H,\br,B) \sim \frac{\Omega(1)}{\Gamma(\rho(X) - \Delta_X(\br))} B(\log B)^{ \rho(X) - \Delta_X(\br)-1}, \quad \text{as } B \to \infty,
\end{equation}
as required for Theorem \ref{thm:Brauer}.

\subsection{Non-vanishing of the leading constant} \label{sec:non-vanishing}
We now verify (\ref{eqn:non-zero}). It is here where subordinate Brauer group elements appear,
and hopefully it should soon become clear to the reader that the study performed earlier
in \S \ref{Sec:subordinate} and \S \ref{sec:sub_tori} was worth the effort.

By Lemma \ref{lem:global_fourier_anisotropic}, the characters
which give rise to the singularity of $Z(s)$ of highest order are exactly the finite collection
of characters $C(T,\R)$, as defined in Theorem \ref{thm:sub_C}. Given (\ref{eqn:HZF_anisotropic}),
applying Theorem \ref{thm:sub_C} we see that we need to consider the sum
$$\sum_{\chi \in \Sub_e(X,\br)/\Be(T)}\widehat{H}(\thorn,\chi;-\sbf).$$
Note that it follows from character orthogonality that 
$$
	\sum_{\chi \in \Sub_e(X,\br)/\Be(T)}\widehat{H}(\thorn,\chi;-\sbf)
	 =\frac{|\Sub_e(X,\br)|}{|\Be(T)|}\int_{T(\Adele_F)_{\br}^{\Sub(X,\br)}}H(t;-s)\mathrm{d}\mu.
$$
Therefore in order to show (\ref{eqn:non-zero}), by (\ref{eqn:rho_Delta}) it suffices to prove that
\begin{equation} \label{eqn:non-zero2}
	\lim_{s \to 1} (s-1)^{\rho(X) - \Delta_X(\br)}
	\int_{T(\Adele_F)_{\br}^{\Sub(X,\br)}} H(t;-s)\mathrm{d}\mu \neq 0.
\end{equation}
If $\Sub(X,\br) \neq \Sub(F(X),\br)$, then Theorem \ref{thm:Sub_Harari} implies that
$T(\Adele_F)_{\br}^{\Sub(X,\br)} \subset T(\Adele_F)_{\br}$ is neither open nor
closed for the product topology and can be quite complicated.
Therefore rather than dealing with it directly, we shall show that the analogue
of (\ref{eqn:non-zero2}) holds for a certain subspace.

\begin{lemma} \label{lem:Sub(X,B)}
	The limit
	$$\lim_{s \to 1} (s-1)^{\rho(X) - \Delta_X(\br)}
	\int_{T(\Adele_F)_{\Sub_e(X,\br)}} H(t;-s)\mathrm{d}\mu,$$
	exists and is non-zero.
\end{lemma}
\begin{proof}
	First note that $T(\Adele_F)_{\Sub_e(X,\br)} \neq \emptyset$; indeed $1 \in T(\Adele_F)_{\Sub_e(X,\br)}$
	as $b(1)=0$ for all $b \in \Sub_e(X,\br)$, by definition. Moreover $\Sub_e(X,\br)$ is finite by
	Lemma \ref{lem:Sub_is_finite}. The integral in the lemma is simply the Fourier transform
	$\widehat{H}(\thorn_{\Sub_e(X,\br)},1;-s)$, so applying Lemma \ref{lem:global_fourier_anisotropic}
	and (\ref{eqn:rho_Delta}) to $\Sub_e(X,\br)$ we deduce that
	$$\lim_{s \to 1} (s-1)^{\rho(X) - \Delta_X(\Sub_e(X,\br))}
	\int_{T(\Adele_F)_{\Sub_e(X,\br)}} H(t;-s)\mathrm{d}\mu \neq 0.$$
	The lemma is proved on noting that $\Delta_X(\Sub_e(X,\br)) = \Delta_X(\br)$, since by definition
	$\Sub_e(X,\br)$ generates the same group of residues as $\br$.
\end{proof}

As $\br \subset \Sub_e(X,\br)$, we obviously have $T(\Adele_F)_{\Sub_e(X,\br)}
\subset T(\Adele_F)_{\br}^{\Sub(X,\br)}$. Hence for $\sigma \in \RR_{>1}$ we obtain
\begin{align*}
	&\lim_{\sigma \to 1^+} (\sigma-1)^{\rho(X) - \Delta_X(\br)}
	\int_{T(\Adele_F)_{\br}^{\Sub(X,\br)}} H(t;-\sigma)\mathrm{d}\mu \\
	&\geq \lim_{\sigma \to 1^+} (\sigma-1)^{\rho(X) - \Delta_X(\br)}
	\int_{T(\Adele_F)_{\Sub_e(X,\br)}} H(t;-\sigma)\mathrm{d}\mu,
\end{align*}
and this latter limit is non-zero by Lemma \ref{lem:Sub(X,B)}.
This  shows (\ref{eqn:non-zero2}) and hence (\ref{eqn:non-zero}),
which completes the proof of Theorem \ref{thm:Brauer}. \qed

\subsection{Calculation of the leading constant} \label{sec:leading_constant}
We now calculate the leading constant in Theorem \ref{thm:Brauer}.
Peyre \cite{Pey95} has formulated a conjectural expression for the leading constant
in the classical case of Manin's conjecture where $\br=0$. This expression
was confirmed for anisotropic tori \cite[Cor.~3.4.7]{BT95}
and takes the shape
\begin{equation} \label{eqn:classical_constant}
	\frac{\alpha(X)\beta(X)\tau(X)}{(\rho(X)-1)!}.
\end{equation}
Here $\alpha(X)$ is a certain rational number defined in terms of the cone
of effective divisors of $X$. For anisotropic 
tori we have $\alpha(X)=1/|\Pic T|$ by \cite[Ex.~2.4.9]{BT95}.
Also $\beta(X)=|\Br X / \Br F|$ and $\tau(X)$ is the Tamagawa number of $X$,
defined as the volume of $\overline{X(F)}$ inside $X(\Adele_F)$
with respect to a certain Tamagawa measure (the factor $\beta(X)$ first appeared in \cite{BT95}).
In what follows the reader
should keep the expression (\ref{eqn:classical_constant}) in mind, as the expression
which we will derive shall bear a striking resemblance to it.

\subsubsection{Virtual Artin $L$-functions}\label{sec:virtual_Artin}
To deal with the leading constant, we shall use
the formalism of virtual Artin $L$-functions.
For a number field $F$, a virtual Artin representation over $F$ is a formal finite sum
$V = \sum_{i=1}^n z_iV_i$
where $z_i \in \CC$ and the $V_i$ are Artin representations.
We define $\rank V = \sum_{i=1}^n z_i \rank V_i$
and let $V^{G_F}=\sum_{i=1}^n z_i V_i^{G_F}$.
The $L$-function of $V$ is defined to be
$$L(V,s)=\prod_{i=1}^n L(V_i,s)^{z_i},$$
where $L(V_i,s)$ is the usual Artin $L$-function associated to $V_i$ (see \cite[Ch.~5.13]{IK04}).
Note that $L(V,s)$ is not an $L$-function in the traditional sense (as in \cite[\S5]{IK04}),
as in general it does not admit a meromorphic continuation to $\CC$. However,
standard properties of Artin $L$-functions imply that $L(V,s)$ admits a holomorphic
continuation with no zeros to the region $\re s \geq 1$, apart from possibly at $s=1$
(the reader who is unfamiliar with complex powers
of $L$-functions is advised to consult \cite[Ch.~II.5.1]{Ten95}).
We have
$$L(V,s) = \frac{c_V}{(s-1)^{r}} + O\left(\frac{1}{(s-1)^{r-1}}\right),$$
as $s \to 1$, where $r = \rank V^{G_F}$ and $c_V \neq 0$.
In this notation we shall write
$$L^*(V,1) = c_V.$$

\subsubsection{A Tamagawa measure}
We now define a Tamagawa measure, which may be viewed
as a generalisation of Peyre's Tamagawa measure \cite{Pey95} to our setting.
This measure is closely related to the measure constructed in \S \ref{sec:Haar_measure_Tamagawa},
though here we choose different convergence factors and also take into account the
adelic metric on $\omega_X$.

Let $\omega$ be an invariant differential form on $T$. For any place $v$ of $F$,
we define the associated local Tamagawa measure to be
$$\tau_v = \frac{|\omega|_v}{||\omega||_v}.$$
This definition is independent of the choice of $\omega$, though depends on the choice of adelic metric on $\omega_X$.
Recalling the construction of $\mu_v$ given in \S \ref{sec:Haar_measure_Tamagawa}, we see that
\begin{equation} \label{eqn:Tamagawa_Fourier_0}
	\tau_v= \frac{c_v\cdot\mu_v}{H_v(\cdot)}.
\end{equation}
In particular we have
\begin{equation} \label{eqn:Tamagawa_Fourier}
	\tau_v(T(F_v)) = c_v \cdot \widehat{H}_v(1,1;-1), \quad
    \tau_v(T(F_v)_\br) = c_v \cdot \widehat{H}_v(\thorn_v,1;-1).
\end{equation}
For the convergence factors, consider the following virtual Artin representation
\begin{equation}\label{def:Pic_br}
	\Pic_\br(\overline{X})_\CC=\Pic(\overline{X})_\CC
	- \sum_{\alpha \in \A}\left(1-\frac{1}{|\res_{\alpha}(\br)|}\right)\Ind_{F_\alpha}^F \CC.
\end{equation}
Here $\Pic(\overline{X})_\CC = \Pic(\overline{X}) \otimes_\ZZ \CC$ and
$\Ind_{F_\alpha}^F$ denotes the induced representation.
Note that when $\br =0$, then (\ref{def:Pic_br}) is simply $\Pic(\overline{X})_\CC$.
Next let $\Pic_\br(X)_\CC= \Pic_\br(\overline{X})^{G_F}_\CC$.
By (\ref{eqn:fundamental}), we may write the corresponding virtual Artin $L$-function as
\begin{equation} \label{eqn:L-function}
	L(\Pic_\br(\overline{X})_\CC,s)= 
	\frac{L(\Pic(\overline{X})_{\CC},s)}{\prod_{\alpha \in \A} \zeta_{F_\alpha}(s)^{(1-1/|	\res_{\alpha}(\br)|)}}
	= \frac{\prod_{\alpha \in \A} \zeta_{F_\alpha}(s)^{1/|\res_{\alpha}(\br)|}}{L(X^*(\Tbar)_\CC,s)}.
\end{equation}
For each place $v \in \Val(F)$ we define
$$ \lambda_v = \left \{
	\begin{array}{ll}
		L_v(\Pic_\br(\overline{X})_\CC,1),& \quad v \text{ non-archimedean}, \\
		1,& \quad v \text{ archimedean}. \\
	\end{array}\right.$$	
In the light of (\ref{eqn:Tamagawa_Fourier}), Lemma \ref{lem:global_fourier_anisotropic}
implies that these are a family of convergence factors, i.e.~the measure
$\prod_{v} \lambda_v^{-1} \tau_{v}$
converges to a measure on $T(\Adele_F)_\br$.
We define the Tamagawa measure on $T(\Adele_F)_\br$ associated to $\br$ to be
\begin{equation} \label{def:Tamagawa}
	\tau_\br = L^*(\Pic_\br(\overline{X})_\CC,1)  \prod_{v \in \Val(F)} \lambda_v^{-1} \tau_{v}.
\end{equation}
Note that we have not included a discriminant factor as in Peyre \cite[Def.~2.1]{Pey95}, since we have normalised
our Haar measure on $\Adele_F$ so that $\vol(\Adele_F/F) = 1$.

\subsubsection{The leading constant}

We now calculate the leading constant in Theorem \ref{thm:Brauer}.

\begin{theorem}\label{thm:Brauer_2}
    Under the same assumptions of Theorem \ref{thm:Brauer} we have
    $$N(U,H,\br,B) \sim c_{X,\br,H} B (\log B)^{\rho_\br(X)-1}, \quad \text{as } B \to \infty,$$
    where $\rho_{\br}(X) = \rank \Pic_\br(X)_\CC$
    and
    $$c_{X,\br,H}=\frac{\alpha(X)\cdot\mathopen|\Sub(X,\br)/\Br F\mathclose|\cdot \tau_\br\left(T(\Adele_F)_\br^{\Sub(X,\br)}\right)}{\Gamma(\rho_\br(X))}.$$
\end{theorem}
\begin{proof}
	The asymptotic formula follows from inserting the definition (\ref{def:Pic_br}) of $\Pic_\br(X)_\CC$
    into Theorem \ref{thm:Brauer}. For $c_{X,\br,H}$, first note that as $T$ is anisotropic, we see that
	$L(X^*(\Tbar),s)$ is holomorphic and non-zero on the half-plane $\re s \geq 1$.
    By (\ref{eqn:HZF_anisotropic}), (\ref{eqn:asym})
	and the work in \S \ref{sec:non-vanishing}, 
	the leading constant $c_{X,\br,H}$ is non-zero and takes the form
    $$\frac{|\Sub_e(X,\br)|}{\Gamma(\rho_\br(X))\vol(T(\Adele_F)/T(F))|\Be(T)|}\lim_{s \to 1} (s-1)^{ \rho(X) - \Delta_X(\br)}
	\int_{T(\Adele_F)_\br^{\Sub(X,\br)}}H(t;-s)\mathrm{d}\mu.$$
	Therefore, on using (\ref{eqn:Tamagawa_Fourier_0}), \eqref{eqn:L-function} and \eqref{def:Tamagawa}
	we see that
    \begin{align*}
    &\lim_{s \to 1} (s-1)^{ \rho(X) - \Delta_X(\br)}
	\int_{T(\Adele_F)_\br^{\Sub(X,\br)}}H(t;-s)\mathrm{d}\mu\\
    & = \lim_{s \to 1} (s-1)^{\rho_\br(X)}\frac{L(\Pic_\br(\overline{X})_\CC,s)}{L(\Pic_\br(\overline{X})_\CC,s)}
	\int_{T(\Adele_F)_\br^{\Sub(X,\br)}}\frac{\mathrm{d}\mu}{H(t)^s} \\
    & = L(X^*(\Tbar),1)\lim_{s \to 1} (s-1)^{\rho_\br(X)}\frac{L(\Pic_\br(\overline{X})_\CC,s)}
		{\prod\limits_{\alpha \in \A} \zeta_{F_\alpha}(s)^{1/|\res_{\alpha}(\br)|}}
	\int_{T(\Adele_F)_\br^{\Sub(X,\br)}}\prod_{v}\frac{L_v(X^*(\Tbar),1)\mathrm{d}\tau_v}{H_v(t_v)^{s-1}} \\
    & = L(X^*(\Tbar),1)\tau_\br\left(T(\Adele_F)_\br^{\Sub(X,\br)}\right).
    \end{align*}	
	Sansuc's duality (\ref{eqn:Be_Sha}), Ono's formula (\ref{eqn:Ono}) and
    the equality $\alpha(X)=1/|\Pic T|$ (see \cite[Ex. 2.4.9]{BT95}) now imply the result.
\end{proof}

\subsubsection{The case where $\Sub(X,\br)=\Sub(F(X),\br)$}
If $\Sub(X,\br) \neq \Sub(F(X),\br)$, then Theorem \ref{thm:Sub_Harari} implies that
$T(\Adele_F)_\br^{\Sub(X,\br)}$ can be quite complicated. 
In particular, the volume appearing in Theorem \ref{thm:Brauer_2} can be difficult to calculate in general.
When this equality holds however, the leading constant takes a more pleasing form.
\begin{lemma}
	Suppose that $\Sub(X,\br)=\Sub(F(X),\br)$. Let 
	$V=V_1 \times_T \cdots \times_T V_r$ be a product of Brauer-Severi schemes over $T$ 
	such that $\langle [V_1],\ldots,[V_r] \rangle = \br$. Then
	$$
		c_{X,\br,H}=\frac{\alpha(X)\cdot|\br|\cdot|\Brnr(F(V)/F)/\Br F|\cdot 
		\tau_\br(\overline{T(F)}^w_\br)}{\Gamma(\rho_\br(X))},
	$$
	where $\overline{T(F)}^w_\br$ denotes the closure of $T(F)_\br$ in
	$T(\Adele_F)_\br$ for the product topology.
\end{lemma}
\begin{proof}
	Theorem \ref{thm:sub_WA} implies
	that $T(\Adele_F)_\br^{\Sub(F(X),\br)}= \overline{T(F)}^w_\br$,
	whereas the equality $$|\Sub(F(X),\br)/ \Br F|=|\br|\cdot|\Brnr(F(V)/F) / \Br F|,$$
	follows from Theorem \ref{thm:CTSD}. Combining these, we obtain the result.
\end{proof}

Theorem \ref{thm:sub_WA} implies that there exists a finite subset $S \subset \Val(F)$ such that
$$\tau_\br(\overline{T(F)}^w_\br) = L^*(\Pic_\br(X)_\CC,1)
\prod_{v \in S} \tau_v\left(\overline{T(F)_\br}^w \bigcap \prod_{v \in S} T(F_v)_\br \right)
\prod_{v \not \in S}\tau_v(T(F_v)_\br).$$
Hence $\tau_\br(\overline{T(F)}^w_\br)$ is a product of
``local densities'' over almost all places, together with a factor which measures the failure of weak approximation for $V$
(in the classical case the corresponding factor measures the failure of weak approximation for $T$ itself).

\subsubsection{The case where $\br \subset \Br X$}
We finish with the special case where $\br \subset \Br X$
(in the notation of Theorem \ref{thm:Brauer}). Note that $\Br_1 X = \Br X$ as $X$ is smooth projective and geometrically rational.
Here $\Delta_X(\br)=0$, in particular $\Sub(X,\br)=\Sub(F(X),\br) = \Br X$
and the measure $\tau_\br$ is the Tamagawa measure $\tau$ defined by Peyre \cite{Pey95}.
The asymptotic formula in Theorem \ref{thm:Brauer} now has the same order of magnitude
as in Manin's conjecture, so one may wonder how the constant $c_{X,\br,H}$
compares with $c_{X,H,\mathrm{Peyre}}$. In such cases, a simple application of 
Theorem \ref{thm:Brauer_2} yields
\begin{equation} \label{eqn:tau/tau}
\lim_{B \to \infty} \frac{N(U,H,\br,B)}{N(U,H,B)} = \frac{\tau(X(\Adele_F)_\br^{\Br X})}{\tau(X(\Adele_F)^{\Br X})}.
\end{equation}
Note however that \eqref{eqn:tau/tau} does not require the full force of Theorem \ref{thm:Brauer_2};
it follows from Theorem \ref{thm:trivial} and holds for toric varieties with respect to not necessarily anisotropic tori, as explained in Section \ref{sec:conclusion}.

We give two types of phenomenon that occur for toric varieties which cannot occur
in the case of projective space considered by Serre \cite{Ser90}. 
Namely, let $U \subset \PP^n$ be a non-empty open and let $\br \subset \Br U$ be a finite
subgroup. We claim that
\begin{equation}
\begin{split} \label{eqn:proj}
\lim_{B \to \infty} \frac{N(U,H,\br,B)}{N(U,H,B)} = 
\begin{cases}
	1, & \br = 0, \\
	0, & \br \neq 0.
\end{cases}
\end{split}
\end{equation}
If $\br = 0$, this is clear. Otherwise, there are two cases. Either $\br$ contains a non-trivial constant Brauer class, in which case $U(F)_\br = \emptyset$, or $\br$ contains a ramified Brauer class. In this latter case when $F=\QQ$ and $\#\br = 2$ the vanishing \eqref{eqn:proj} follows from \cite[Thm.~2]{Ser90}; the result for general $F$ and $\br$ follows from \cite[Thm.~2.4]{LS15}.


For toric varieties however, we obtain a much wider range of behaviour than \eqref{eqn:proj},
as \eqref{eqn:tau/tau} and following lemma demonstrate.

\begin{lemma} \label{lem:Peyre}
	\begin{enumerate}
		\item[]
		\item There exist $T$ and $\br \subset \Br X$ over $F$  with $\br \neq 0$, such that
		$$\frac{\tau(X(\Adele_F)_\br^{\Br X})}{\tau(X(\Adele_F)^{\Br X})} = 1.$$
		\item There exist $T$ and $\br \subset \Br X$ over $F$, such that
		$$0 < \frac{\tau(X(\Adele_F)_\br^{\Br X})}{\tau(X(\Adele_F)^{\Br X})} < 1.$$
	\end{enumerate}
\end{lemma}
\begin{proof}
	For the proof, we  use the well-known fact that any torus $T$ over any field 
	admits a smooth projective equivariant compactification $X$ (see \cite[Cor.~1]{CTHS05}).

	For $(1)$, choose $T$ such that $\Be(T) \neq 0$ and take $\br = \Be(T)$.
	We have $\br \subset \Br X$ by \eqref{seq:Br_Sha}. Note that by  (\ref{eqn:Be_Sha}),
	the condition $\Be(T) \neq 0$ is equivalent to $\Sha(T) \neq 0$. 
	This occurs, for example, for the
	norm one torus over $\QQ$ for the field extension $\QQ(\sqrt{13},\sqrt{17})/\QQ$
	\cite[p.~224]{CTS77}. Similar examples over	every number field 
	can be constructed using Chebotarev's density theorem.
	By (\ref{eqn:CFT}) and \eqref{def:Be} we have $X(F)_\br = X(F)$,
	hence $\tau(X(\Adele_F)_\br^{\Br X})=\tau(X(\Adele_F)^{\Br X})$.
	Geometrically, we obtain a product
	of Brauer-Severi schemes over $X$ which is Zariski locally trivial over each $F_v$, yet not over $F$.

	For $(2)$, we choose $T$ and $\br$ such that $X(F)_\br \neq X(F)$ but $X(F)_\br \neq \emptyset$.
	Over $\QQ$, one may take the norm one torus for the
	field extension $\QQ(\sqrt{2},\sqrt{3})/\QQ$, as here
	not every rational point is Brauer equivalent to $1 \in T(\QQ)$
	(see \cite[p.~224]{CTS77} for this example and \cite[\S7]{CTS77}
	for definitions -- as before, similar examples exist over any number field).
	Note that such a $T$ must necessarily fail weak approximation.
	As $\br \subset \Br X$, there exists a finite set of places $S$ 
	such that $X(F_v)_\br = X(F_v)$ for all $v \not \in S$.
	Moreover, as $\br$ is finite, the complement of $\smash{X(\Adele_F)_\br^{\Br X}}$
	in $\smash{X(\Adele_F)^{\Br X}}$ is open and closed, hence has positive measure 
	with respect to $\tau$. This gives the required example.
\end{proof}


\end{document}